\documentclass[a4, 11pt]{amsart}
\usepackage[T1]{fontenc}
\usepackage{amssymb,latexsym}
\usepackage{color}
\usepackage{graphicx}
\usepackage{enumitem}
\usepackage[left=35mm, right=35mm, top=30mm, bottom=30mm, includefoot, includehead]{geometry}
\theoremstyle{plain}
\newtheorem{theorem}{Theorem}[section]
\newtheorem{corollary}[theorem]{Corollary}
\newtheorem{proposition}[theorem]{Proposition}
\newtheorem{lemma}[theorem]{Lemma}
\theoremstyle{definition}
\newtheorem{definition}[theorem]{Definition}
\newtheorem{example}[theorem]{Example}

\newtheorem{remark}[theorem]{Remark}
\newtheorem{question}[theorem]{Question}
\newtheorem*{theo}{Theorem}

\newcommand{\enm}[1]{\ensuremath{#1}}          %
\newcommand{\CC}{\enm{\mathbb{C}}}

\newcommand{\NN}{\enm{\mathbb{N}}}

\newcommand{\PP}{\enm{\mathbb{P}}}

\newcommand{\Spec}{\mathrm{Spec}}

\newcommand{\DS}{\displaystyle}

\renewcommand{\phi}{\varphi}
\renewcommand{\theta}{\vartheta}
\renewcommand{\epsilon}{\varepsilon}

\renewcommand{\to}[1][]{\xrightarrow{\ #1\ }}

\newcommand{\old}[1]{}

\date{}
\title[]{Vanishing Hessian, wild forms and their border VSP}
\author{Hang (Amy) Huang, Mateusz Micha\l{}ek, and Emanuele Ventura}

\address{Dept. of Mathematics, Texas A\&M University,
College Station, TX 77843-3368, USA}
\email{hhuang@math.tamu.edu}

\address{Max Planck Institute for Mathematics in the Sciences, Leipzig, Germany}
\email{mateusz.michalek@mis.mpg.de}

\address{University of Bern, Mathematical Institute, Sidlerstrasse 5, 3012 Bern, Switzerland}
\email{emanueleventura.sw@gmail.com, emanuele.ventura@math.unibe.ch}


\begin{document}

\maketitle

\begin{abstract}
Wild forms are homogeneous polynomials whose smoothable rank is strictly larger than their border rank. The discrepancy between these two ranks is caused by the difference between the limit of spans of a family of zero-dimensional schemes  and the span of their flat limit. For concise forms of minimal border rank, we show that the condition of vanishing Hessian is equivalent to being wild. This is proven by making a detour through structure tensors of smoothable and Gorenstein algebras. The equivalence fails in the non-minimal border rank regime. We exhibit an infinite series of minimal border rank wild forms of every degree $d\geq 3$ as well as an infinite series of wild cubics. Inspired by recent work on border apolarity of Buczy\'nska and Buczy\'nski, we study the border varieties of sums of powers $\underline{\mathrm{VSP}}$ of these forms in the corresponding multigraded Hilbert schemes. 
\end{abstract}

\section{Introduction}

Notions of ranks abound in the literature, perhaps because of their natural appearance in the realms of algebra and geometry, and in numerous applications thereof;
see \cite{ik, Lands} and references therein for an introduction to the subject.

These ranks vastly generalize matrix rank and yet they are very classical, dating back to the pioneering work of Sylvester. His work featured Waring ranks of binary forms; see \cite{ik} for a historical account on the subject. Since then, ever growing research efforts have been devoted to understanding ranks with respect to some special projective varieties $X$ of interest. Last decades have witnessed steady progress on tensor and Waring ranks, i.e. the cases when
the projective varieties are the classical Segre and Veronese varieties. 

These results have been developed in parallel in their geometric and algebraic aspects. The first are naturally related to secant varieties of $X$ \cite[Chapter 1]{russo}, whereas the second to Macaulay's theory of apolarity and inverse systems \cite[\S 1.1]{ik}. 

Interestingly, scheme-theoretic versions of $X$-ranks have been introduced and studied as well. These latter ones take into account more general zero-dimensional schemes, besides the reduced zero-dimensional ones featured in the $X$-ranks. This more general framework naturally leads to new notions of $X$-rank: the {\it smoothable $X$-rank}, and the {\it cactus $X$-rank}; the latter was originally called {\it scheme length} \cite[Definition 5.1]{ik}. We recall their definitions in \S\ref{prelim}. 

One subtle phenomenon is that, for special points, perhaps unexpectedly smoothable ranks may be larger than border ranks. This discrepancy is caused by the difference between the limit of spans of a family of zero-dimensional schemes and the span of their flat limit. The difference between smoothable and border ranks does not appear for general points (forms) of fixed border rank. Therefore, it is a natural and interesting problem to investigate the structure of the instances where these two differ.  

As far as we know, Buczy\'nska and Buczy\'nski \cite{bb15} were the first authors to bring the difference between these ranks to attention. They introduced the notion of {\it wild forms}, i.e. those whose smoothable rank is strictly larger than their border rank. They gave one such a form \cite[\S 4]{bb15}, up to concise minimal border rank direct summands. 

On another direction, in recent groundbreaking work, Buczy\'nska and Buczy\'nski \cite{bb19} expanded the apolarity theory of $X$-ranks to {\it border apolarity}, which is devised to provide information about border $X$-ranks. Along the way, they introduced the {\it border varieties of sums of powers} $\underline{\mathrm{VSP}}$, mirroring the classical varieties of sums of powers \cite{Mukai1, rs00}. 

Inspired by \cite[\S 5.3]{bb19}, we establish a new result on wild forms. To state it, let $V$ be a complex finite-dimensional vector space; given a form $F\in S^dV^{*}$, let $\mathrm{Hess}(F)$ denote the determinant of its Hessian matrix. Forms with identically vanishing Hessian have many remarkable geometric and algebraic properties; see \cite[Chapter 7]{russo} for a detailed and updated exposition. 

Forms with vanishing Hessian were originally studied by Hesse in two classical papers \cite{Hesse1859, Hesse1851}, where the author tried to prove that these homogeneous polynomials are necessarily not concise (or, in more geometric terms, that the hypersurfaces they define 
are cones). Thereafter, in their important work \cite{GN1876}, Gordan and Noether showed that Hesse's claim is true in the regime of at most four variables, whereas there exist infinitely many counterexamples afterwords. The easiest counterexample is perhaps the {\it Perazzo cubic hypersurface} \cite{perazzo} (in Perazzo's words ``{\it un esempio semplicissimo}'', i.e. ``a very easy example''), which appears in \cite[\S 4]{bb15} as an instance of wild cubic form: that is the point of departure of our article. Our main result is Theorem \ref{hessvswild}, which connects wild forms and vanishing Hessian, following the lines paved by Ottaviani's remark \cite[Remark 5.1]{bb19}. (For definitions of the ranks involved in the statement, see \S\ref{prelim}.)

\begin{theo}\emph{
Let $d\geq 3$ and $F\in S^d V^{*}$ be a concise form of minimal border rank. Then:
\[
\mathrm{Hess}(F) = 0 \ \ \Longleftrightarrow \mathrm{cr}(F) > \underline{{\bf r}}(F)  \ \ \Longleftrightarrow  \mathrm{sr}(F) > \underline{{\bf r}}(F) \ \ \Longleftrightarrow F  \mbox{ is wild}.
\]
}
\end{theo}

The equivalence fails when assuming non-minimal border rank: there exist concise wild forms of non-minimal border rank whose Hessian is not vanishing, see Example \ref{wildwithnozerohess}. 

Let $T = \bigoplus_{d\geq 0} S^d V$. Given $F\in S^d V^{*}$, let $\mathrm{Ann}(F)\subset T$ be the annihilator or apolar ideal of $F$, see \S\ref{prelim}. 
Let $A = T/\mathrm{Ann}(F)$ be the Artinian Gorenstein $\CC$-algebra of (a concise) $F$. Let $\lbrace \alpha_i^{(k)}\rbrace$ be a basis of $A_k$. Then the determinant
\[
\mathrm{Hess}^k(F) = \mathrm{det}\left(\alpha_i^{(k)}\alpha_j^{(k)} F\right)
\]
is the {\it $k$-th Hessian} of $F$. By definition, $\mathrm{Hess}^1(F) = \mathrm{Hess}(F)$. 

The {\it Strong Lefschetz Property} (SLP) of $A$ is characterized in terms of these higher Hessians of $F$: $A$ has $\mathrm{SLP}$ if and only if $\mathrm{Hess}^k(F)\neq 0$ for every $k=1,\ldots, \lfloor d/2\rfloor$ \cite[Theorem 3.1]{MW}.  Then the result above reads: 

\begin{theo}\emph{
Let $d\geq 3$ and let $F\in S^d V^{*}$ be a concise form of minimal border rank. Then:
\[
F  \mbox{ is wild} \ \ \Longrightarrow \ \ T/\mathrm{Ann}(F) \mbox{ does not have }  \mathrm{SLP}.
\]
For $d=3$ and $d=4$, this is an equivalence.
}
\end{theo}
To prove our main result we rely on the interplay between tensors and algebras. In particular, we prove the following: 

\begin{theo}
A finite-dimensional $\CC$-algebra $A$ is Gorenstein if and only if its $3$-way structure tensor (i.e.~the tensor associated to the multiplication map $A\times A\rightarrow A$) is symmetric if and only if its $d$-way structure tensor for some $d>2$ (equivalently, for every $d$) is symmetric.
\end{theo}

We also provide a possible extension of the results of Bl\"{a}ser and Lysikov \cite{BL} to $d$-way tensors: 

\begin{theo}
Let $T$ be a symmetric $d$-way tensor of minimal border rank. Suppose there exists a contraction $T(\ell^{\otimes d-2})$ which is a full-rank symmetric matrix. Then $T$ is the structure tensor of a smoothable Gorenstein algebra.
\end{theo}

As an application of our main result, we exhibit two infinite series of (concise) wild forms. 
In \S\ref{higherdeg}, we give a series of wild forms $G_d$ of every degree $d\geq 3$, and in \S\ref{wildcubics} a series of wild cubics $F_n$. Both
of them are of minimal border rank. In particular, this shows the next 

\begin{theo}\emph{
There exist concise minimal border rank wild forms of any degree $d\geq 3$.
}
\end{theo}

Employing Buczy\'nska-Buczy\'nski's border apolarity theory, we offer a study of border varieties of sums of powers $\underline{\mathrm{VSP}}$'s
of these forms in the corresponding multigraded Hilbert schemes. To our knowledge, this is the first attempt to describe such varieties for some forms. 
For the first series $G_d$, we show that they are projective spaces; see Theorem \ref{vspprojectivespaces}. 

For the second series $F_n$ and $n\geq 10$, in Theorem \ref{vspreducible} we prove that they are reducible. This is achieved by relying on the (usual) Hilbert schemes of zero-dimensional schemes on a {\it chain of lines}, see \S\ref{wildcubics}. We point out that this result on reducibility is also motivated by the fact that establishing this property is usually a delicate and interesting issue even in the context of the classical varieties of sums of powers $\mathrm{VSP}$'s.

\vspace{2mm}

\noindent {\bf Structure of the paper.}
\vspace{1mm}

In \S\ref{prelim}, we introduce notation and recall the definitions of ranks we need throughout the article. 
\S\ref{mainsec} is devoted to the proof of the first part of our main result, i.e., vanishing Hessian implies wild for concise forms of minimal border rank, Theorem \ref{hessvswild}. 

In \S\ref{cubicssmoothablealgebras}, we make a detour through structure tensors of smoothable and Gorenstein algebras to 
establish Theorem \ref{maincubics}. The latter gives the remaining part of our main result.

In \S\ref{seclimsch}, we give the definition of a limiting scheme of a border rank decomposition and its relations with $\underline{\mathrm{VSP}}$. Theorem \ref{satidealsvsp} recalls that the saturation of an ideal in $\underline{\mathrm{VSP}}$ is the ideal of a limiting scheme of a border rank decomposition. 
Theorem \ref{correspondence} shows the correspondence between ideals and border decompositions, in the regime of minimal border rank. 

In \S\ref{higherdeg}, we introduce the infinite series of concise degree $d$ forms $G_d$. We show that they are wild in Corollary \ref{infiniteseriesdeg}. Moreover, 
we prove that their $\underline{\mathrm{VSP}}$'s are isomorphic to projective spaces; this is achieved in Theorem \ref{vspprojectivespaces}. 

In \S\ref{wildcubics}, we introduce the infinite series of concise cubics $F_n$. Corollary \ref{infiniteseries} states that they are wild. We show that when $n\geq 10$, 
their $\underline{\mathrm{VSP}}$'s are reducible; see Theorem \ref{vspreducible}. 

\vspace{2mm}

\noindent {\bf Acknowledgements.}
\vspace{1mm}

We thank Edoardo Ballico, Weronika Buczy\'nska, Jaros\l{}aw Buczy\'nski, J.M. Landsberg, Giorgio Ottaviani, and Francesco Russo for very useful discussions. 
We thank Giorgio Ottaviani for his remark \cite[Remark 5.1]{bb19}: this stimulated the curiosity that led to the present article. 
We thank an anonymous referee for useful comments and for posing interesting and stimulating questions. The third author is supported by the grant NWO Den Haag no. 38-573 of Jan Draisma. 

\section{Preliminaries}\label{prelim}

Here we introduce notation and definitions we use throughout the paper. 
We work over the complex numbers. Let $V \cong \CC^{n+1}$ and $V^* = \langle x_0, \ldots, x_n\rangle$. Let $\PP^n=\PP(V)$ denote the projectivization of $V$. 
Let $S = S^{\bullet} V^{*} \cong \CC[x_0,\ldots, x_n]$ be its homogeneous coordinate ring, and $T = \CC[y_0,\ldots, y_n]$ be its dual ring, i.e. $T$ acts by 
differentiation on $S$ with $y_i \circ x_j = \delta_{i,j}$. 

For a homogeneous ideal $\mathcal J\subset T$, let $\mathcal J_d$ denote its degree $d$ homogeneous component. For $\mathcal J\subset T$,
let $\mathcal J^{sat}$ denote its saturation. The {\it Hilbert function}
of $\mathcal J$ is the numerical function
\[
\mathrm{HF}(T/\mathcal J, d) = \dim S^d V - \dim \mathcal{J}_d = \dim\left(T/\mathcal J\right)_{d}. 
\]

Given a form $F\in S^d V^*$, $\mathrm{Ann}(F)\subset T$ denotes its annihilator or apolar ideal 
\[
\mathrm{Ann}(F) = \left\lbrace h \in T \ | \ h\circ F = 0\right\rbrace.
\]
The algebra $T/\mathrm{Ann}(F)$ is a graded Artinian Gorenstein $\CC$-algebra; see \cite[\S 2.3]{ik} or \cite[Theorem 7.2.15]{russo}. 

Let $N_d = \binom{n+d}{d}-1$ and $X = \nu_d(\PP^n)\subset \PP^{N_d}$ be the $d$-th Veronese embedding of $\PP^n$. We only consider ranks with respect to the Veronese variety, although the ensuing definitions may be more generally introduced for any projective variety. 

Let $R$ be a zero-dimensional scheme over $\CC$; then $R = \mathrm{Spec}(A)$ for some finite-dimensional $\CC$-algebra $A$. The {\it length} of $R$ is $\mathrm{length}(R) = \dim_{\CC} A$. 

For a zero-dimensional scheme $R\subset \PP^n=\PP(V)$, let $\langle R\rangle$ denote its span, i.e. 
\[
\langle R\rangle = \PP\left(\left(V^*/\mathcal I_1\right)^{*}\right)\subset \PP^{n},
\] 
where $\mathcal I$ is the saturated ideal defining $R$ in $\PP^{n}$. 

\begin{definition}[{\bf Border rank}]
For a point $F\in \PP^{N_d}$, the {\it border rank} of $F$ is the minimal integer $r$ such that $F\in \sigma_r(X)$, the $r$-th secant variety of $X$. The border rank of $F$ is denoted $\underline{{\bf r}}(F)$. 
\end{definition}

\begin{definition}
A form $F\in S^d V^{*}$ is {\it concise} if its annihilator $\mathrm{Ann}(F)$ does not contain linear forms.\end{definition}

\begin{definition}\label{minbr}
A form $F\in S^d V^{*}$ is said to be of {\it minimal border rank} when 
\[
\underline{{\bf r}}(F) = \dim V.
\]
\end{definition}

\begin{remark}
The border rank of a concise form $F\in S^d V^{*}$ satisfies $\underline{{\bf r}}(F)\geq \dim V$. This explains the adjective {\it minimal}
in Definition \ref{minbr}. 
\end{remark}

In general, given $F\in S^d V^{*}$, it might be challenging to produce a border rank decomposition $F = \lim_{t\rightarrow 0}\frac{1}{t^s}\left(L_1(t)^d + \ldots  +L_{\underline{{\bf r}}(F)}(t)^d\right)$. To determine border ranks for our infinite series of forms $G_d$ and $F_n$, we employ a useful criterion:

\begin{proposition}[{\bf \cite[Proposition 2.6]{bb15}}]\label{tangentcones}
Let $X$ be as above. 
Suppose there exist points $z_1,\ldots, z_r\in X$ such that $\dim \langle z_1,\ldots, z_r\rangle< r-1$. Then 
the span of the affine cones of Zariski tangent spaces at these points is contained in the $r$-th secant variety $\sigma_r(X)$, i.e. 
\[
\langle \PP \widehat{T}_{z_1}, \ldots, \PP \widehat{T}_{z_r}\rangle\subset \sigma_r(X).  
\]
\end{proposition}

We now recall the scheme-theoretic ranks attached to $X$. 

\begin{definition}[{\bf Smoothable rank}]
The {\it smoothable rank} of $F\in \PP^{N_d}$ is the minimal integer $r$ such that there exists a finite scheme $R\subset X$ of length $r$ which is {\it smoothable}
(in $X$) and $F\in \langle R\rangle$. Equivalently, there exists a finite smoothable scheme $R\subset X$ of length $r$ whose saturated ideal satisfies $\mathcal J_R\subset \mathrm{Ann}(F)$. (This is in analogy with the classical {\it Apolarity lemma} \cite[Lemma 1.15]{ik}.) The smoothable rank of $F$ is denoted $\mathrm{sr}(F)$. 
\end{definition}

\begin{remark}
Smoothable and border ranks satisfy $\underline{{\bf r}}(F)\leq \mathrm{sr}(F)$; see \cite[\S 2.1]{bb15}. Equality holds
for most of the points on a secant variety. The difference arises in general from the {\it failure} of the equality 
\[
\langle \lim_{t\rightarrow 0} R(t)\rangle = \lim_{t\rightarrow 0} \langle R(t) \rangle,
\]
where $R(t)$ is a family of zero-dimensional schemes over the base $\mathrm{Spec}(\CC[t^{\pm 1}])$ and $\lim_{t\rightarrow 0} R(t)$ denotes its
flat limit; see for instance \cite[II.3.4]{eh}.
\end{remark}

\begin{definition}[{\bf Cactus rank}]
The {\it cactus rank} of $F\in \PP^{N_d}$ is the minimal integer $r$ such that there exists a finite scheme $R\subset X$ of length $r$  such that $F\in \langle R\rangle$. 
Equivalently, there exists a finite scheme $R\subset X$ of length $r$ whose saturated ideal satisfies $\mathcal J_R\subset \mathrm{Ann}(F)$. (This is in analogy with the classical {\it Apolarity lemma} \cite[Lemma 1.15]{ik}.) The cactus rank of $F$ is denoted $\mathrm{cr}(F)$. 
\end{definition}

The cactus rank was originally called {\it scheme length} \cite[Definition 5.1]{ik}. 

\begin{remark}
From definitions, one has $\mathrm{cr}(F)\leq \mathrm{sr}(F)$. Moreover, $\mathrm{cr}(F)$ and \underline{{\bf r}}(F) are incomparable. 
Theorem \ref{hessvswild} produces infinitely many examples where $\mathrm{cr}(F) > \underline{{\bf r}}(F)$; however one has $\mathrm{cr}(F) < \underline{{\bf r}}(F)$ as well in several instances; see e.g. \cite{br13}. 
\end{remark}

\begin{definition}
A form $F\in S^d V^*$ is a {\it form with vanishing Hessian} if the determinant of its Hessian matrix 
$\mathrm{Hess}(F) = \mathrm{det}\left(\left[\frac{\partial^2 F}{\partial x_i\partial x_j}\right]\right)$ vanishes identically. 
\end{definition}
See \cite[Chapter 7]{russo} and references therein for a complete introduction to several remarkable algebraic and geometric properties that forms with vanishing Hessian possess.

For any form $F\in S^d V^{*}$, Buczy\'nska and Buczy\'nski \cite[\S 4.1]{bb19} introduced the {\it border variety of sums of powers} $\underline{\mathrm{VSP}}(F,\underline{{\bf r}}(F))$. These varieties live in {\it multigraded Hilbert schemes}, which were introduced by Haiman and Sturmfels \cite{hs}. We now recall
the definition of an irreducible component of the multigraded Hilbert scheme we are concerned with; see \cite[\S 3]{bb19} for a detailed discussion.

\begin{definition}
An ideal $\mathcal J\subset T$, whose Hilbert polynomial is equal to $r\in \NN$, is said to have a {\it generic Hilbert function} if its Hilbert function 
satisfies 
\[
\mathrm{HF}(T/\mathcal J, d) = \min\lbrace r, \dim S^d V \rbrace, \mbox{ for } d\geq 0.
\]
Let $\mathrm{Slip}_{r,\PP^n}$ be the irreducible component of the multigraded Hilbert scheme $\mathrm{Hilb}^{h_r,\PP^n}_T$ containing the radical ideals of $r$ distinct points with a generic Hilbert function. Therefore, every ideal in $\mathrm{Slip}_{r,\PP^n}$ has a generic Hilbert function being a flat limit of such ideals. \end{definition}

\begin{definition}
Let $F\in S^d V^{*}$. The {\it border variety of sums of powers}, or {\it border VSP}, of $F$ is 
\[
\underline{\mathrm{VSP}}(F , r) = \left\lbrace \mathcal J\in \mathrm{Slip}_{r,\PP^n} \ | \ \mathcal J\subset \mathrm{Ann}(F)\subset T \right\rbrace.
\]
One case of interest is when $r= \underline{{\bf r}}(F)$.
\end{definition}

\section{Wildness}\label{mainsec}

\begin{definition}
A form $F\in S^d V^{*}$ is {\it wild} if $sr(F) > \underline{{\bf r}}(F)$. 
\end{definition}

\begin{remark}
A consequence of wildness of a form $F$ is that all the ideals in $\underline{\mathrm{VSP}}(F, \underline{{\bf r}}(F))$ are {\it not saturated}. 
\end{remark}

Before proving the next result, we introduce another piece of notation. Let $W\subset U$ be finite-dimensional vector spaces. Then $W^{\perp} = \lbrace h\in U^{*} \ | \ h(w) = 0 \mbox{ for all } w\in W\rbrace$ is the  {\it annihilator} of $W$. 

\begin{proposition}\label{transcdeg}

Let $R\subset \PP^n$ be a projective scheme defined by the saturated ideal $\mathcal J$ such that $\langle R\rangle = \PP^n$.  
Then, for any $d\geq 1$, the affine variety $\CC[\left(\mathcal{J}_d\right)^{\perp}]$ has dimension $n+1$, i.e. the linear space $\left(\mathcal{J}_d\right)^{\perp}$ is spanned by at least $n+1$ algebraically independent forms. 

\begin{proof}
Choose a linear form $z$ such that $\mathcal J:z = \mathcal J$, i.e. $\mathcal J$ is saturated with respect to $z$, or equivalently,
there is no associated prime containing $z$. Such a linear form exists by assumption. 

Let  $S= \CC[z,x_1,\ldots,x_n]$ be the homogeneous coordinate ring of the ambient projective space. Up to change of basis, we can present the homogeneous saturated ideal of $R$ as follows: 
\[
\mathcal J = \langle z^{d_1-1}\ell_1 + h_1, \ldots, z^{d_k-1}\ell_k + h_k\rangle,
\]
where the $\ell_i$ are (possibly zero) linear forms and the $h_j$ are forms of degree $d_j$ that are quadratic in the variables $x_i$. 

Let $W = \mathcal{J}_d$. Note that $z^{d}\in W^{\perp}$ and let $V = \langle z^{d-1}x_1,\ldots, z^{d-1}x_n\rangle$. By assumption $\mathcal{J}_1=0$. 
Notice that $V\cap W = 0$. Indeed, if $z^{d-1}l(x_1,\ldots, x_n) \in W\subset \mathcal{J}$ with $l$ a linear form, then $l(x_1,\ldots,x_n)\in \mathcal{J}$, which is in contradiction with our assumption.

The condition $V\cap W = 0$ implies that we have a surjection
\[
S_d^{*}\supset W^{\perp} \twoheadrightarrow S_d^{*}/V^{\perp} = V^{*}\subset S_d^{*}. 
\]
Thus we can lift the basis of $V^{*}\subset S_d^{*}$ consisting of the vectors $z^{d-1}x_1,\ldots, z^{d-1}x_n$ (by abuse of notation, the duals of $z^{d-1}x_i$ are denoted in the same way) to an independent set in $W^{\perp}$. Therefore: 
\[
\left(\mathcal{J}_d\right)^{\perp} \supseteq \langle z^d, z^{d-1}x_1+g_1, \ldots, z^{d-1}x_n + g_n\rangle. 
\]

To show that the forms on the right-hand side are algebraically independent, we dehomogenize them and look at the affine map they induce: 
\[
\phi: \CC^n \longrightarrow \CC^n,
\]
\[
(x_1,\ldots, x_n)\longmapsto (x_1 + \tilde{g_1}, \ldots, x_n + \tilde{g_n}).
\]

Notice that the differential of $\phi$ at the origin ${\bf 0}\in \CC^n$ is the identity. Thus there is a local isomorphism between tangent spaces and so the dimension 
of the image is $n$. Therefore the dimension of the affine variety  $\CC[\left(\mathcal{J}_d\right)^{\perp}]$ is $n+1$. This is equivalent to 
the fact that the linear space $\left(\mathcal{J}_d\right)^{\perp}$ is spanned by at least $n+1$ algebraically independent degree $d$ forms. 
\end{proof}
\end{proposition}

\begin{theorem}\label{hessvswild}
Let $d\geq 3$ and $F\in S^d V^{*}$ be a concise form of minimal border rank. Then:
\[
\mathrm{Hess}(F) = 0 \ \ \Longrightarrow \ \ \mathrm{cr}(F)>\underline{{\bf r}}(F)=n+1  \ \ \Longrightarrow \ \ F  \mbox{ is wild}. 
\]
\begin{proof}
Let $W = \langle \frac{\partial}{\partial x_i} F\rangle\subset \PP(S^{d-1} V^{*})$ be the linear space spanned by the first 
derivatives of $F$. Let $\mathcal I =  \langle\mathrm{Ann}(F)_{d-1}\rangle$ be the homogeneous ideal generated by the degree $(d-1)$
homogeneous piece of the annihilator of $F$. 

Note that $W^{\perp} = \mathcal I_{d-1}$. Indeed, the inclusion $\mathcal I_{d-1}\subset W^{\perp}$ is clear by definition. To see
the converse, let $h\in W^{\perp}$. Hence $h\circ \left(\frac{\partial}{\partial x_i} F\right)=0$ for every $i$; the latter implies $\frac{\partial}{\partial x_i}\left(h\circ F\right)=0$
for every $i$, where $h\circ F$ is a linear form. Thus $h\circ F=0$ and so $h\in \mathcal I_{d-1}$. 

Let $R^{W}$ be the projective scheme in $\PP^n$ defined by $\mathcal I^{sat}$, which is a priori possibly empty.  As $\left(\mathcal I^{sat}\right)_{d-1} \supseteq \mathcal I_{d-1}$, we have: 
\[
W = W^{\perp\perp} = \left(\mathcal I_{d-1}\right)^{\perp} \supseteq \left((\mathcal I^{sat})_{d-1}\right)^{\perp}. 
\]

Assume the saturated ideal $\mathcal I^{sat}$ does not contain any linear form. By definition, this is equivalent to the subscheme $R^{W}\subset \PP^n$ spanning the whole $\PP^n$.
By Proposition \ref{transcdeg}, $\left((\mathcal I^{sat})_{d-1}\right)^{\perp}$ is spanned by at least $n+1$ algebraically independent forms. However, $W$ has dimension exactly $n+1$, so $W = \left((\mathcal I^{sat})_{d-1}\right)^{\perp}$ and a basis of $W$ consists of $n+1$ algebraically independent forms. Therefore, the derivatives of $F$ must be algebraically independent and so $\mathrm{Hess}(F) \neq 0$; see \cite[\S 7.2]{russo}. This shows that whenever $F$ is concise and has vanishing Hessian, $\mathcal I^{sat}$ must contain a linear form.

Suppose $F$ is concise with $\mathrm{Hess}(F) = 0$ and of minimal border rank. Then the Hilbert function of $\mathrm{Ann}(F)$ is as follows: 
\[
\mathrm{HF}(T/\mathrm{Ann}(F)): 1 \ \ (n+1) \ \ \ldots \ \ (n+1) \ \ 1
\]

Now, we show by contradiction the first implication in the statement. Suppose the cactus rank of $F$ satisfies $\mathrm{cr}(F)\leq n+1$.

Let $\mathcal J\subset \mathrm{Ann}(F)$ be any saturated ideal evincing the cactus rank of $F$, i.e. the zero-dimensional scheme defined by $\mathcal J$ has degree $\mathrm{cr}(F)$. Since its Hilbert function $\mathrm{HF}(T/\mathcal J)$ is non-decreasing until it stabilizes to the constant polynomial $\mathrm{cr}(F)\in \NN$ \cite[Theorem 1.69]{ik}, one has
\[
\dim\left(T/\mathcal J\right)_{d-1}\leq n+1 = \dim\left(T/\mathcal I\right)_{d-1}. 
\]
On the other hand, $\mathcal J_{d-1}\subset \mathcal I$ and so 
\[
\dim\left(T/\mathcal J\right)_{d-1}\geq \dim\left(T/\mathcal I\right)_{d-1}. 
\]
The inequalities imply $\mathcal J_{d-1} = \mathcal I_{d-1}$. 

Now, $\mathcal I^{sat}\subset \mathcal J^{sat} = \mathcal J$. Hence $\mathcal J$ contains a linear form, i.e. 
\[
\dim\left(T/\mathcal J\right)_1\leq n. 
\]
On the other hand, since $\mathcal J\subset \mathrm{Ann}(F)$, one has:
\[
n+1 = \dim\left(T/\mathrm{Ann}(F) \right)_1\leq \dim\left(T/\mathcal J\right)_1\leq n,
\]
which is a contradiction. Therefore $\mathrm{sr}(F)\geq \mathrm{cr}(F) > n+1 = \underline{{\bf r}}(F)$. Hence $F$ is wild. 
\end{proof}

\end{theorem}

\begin{remark}
Keep the notation from Theorem \ref{hessvswild}. It is clear that for $d=2$ and any $n\geq 1$, the condition $\mathrm{Hess}(F)=0$ is equivalent to $F$ being not concise. Also, for $d=3$ and $n\leq 3$, 
$\mathrm{Hess}(F)=0$ is equivalent to $F$ being not concise \cite[Theorem 7.1.4]{russo}. A complete classification is known up to $n\leq 6$; see \cite[\S 7.6]{russo}
for a detailed discussion. 
\end{remark}

\begin{remark}[{\bf Limits of catalecticants and vector bundles}]\label{limits}

Let $F\in S^d V^{*}$ be a form with annihilator $\mathcal I = \mathrm{Ann}(F)\subset T$. For each $0\leq i\leq d$, $F$ induces 
a linear map
\[
\mathrm{Cat}_{i, d-i}: S^{i} V \longrightarrow S^{d-i} V^{*},\ \ h\mapsto h\circ F. 
\]
The map $\mathrm{Cat}_{i,d-i}$ is the $i$-th {\it catalecticant}. Its kernel satisfies the equality $\mathrm{Ker}(\mathrm{Cat}_{i, d-i}) = \mathcal I_{i}$,
and therefore $\mathrm{HF}(T/\mathcal I, i)$ is the rank of the matrix representing the $i$-th catalecticant. 

The rank of any catalecticant is a lower bound to  $\underline{{\bf r}}(F)$, the border rank of $F$. Ga\l{}\k{a}zka \cite{Ga17} showed
that degeneracy conditions of vector bundles vanish on cactus varieties. In our context, this means that the rank of any catalecticant (and so any value $\mathrm{HF}(T/\mathcal I, i)$) gives a lower bound to $\mathrm{cr}(F)$, the cactus rank of $F$. It is inherently difficult to detect the true cactus rank when it is not witnessed by the rank of some catalecticant (which is the case, in the situation of Theorem \ref{hessvswild}), because
one has to choose suitable (strictly contained) linear subspaces of the vector spaces $\mathcal I_i$, which is very hard in practice. Furthermore, one has to construct linear subspaces so that the resulting ideal of a zero-dimensional scheme has the minimal degree allowed. 
In conclusion, for explicit forms (and, even more so, for a sequence of forms), calculating cactus rank in several variables is usually a daunting task. 

Moreover, the smoothable rank of a form is a priori even more difficult to compute explicitly: the general obstacle to overcome is being able to recognize ideals of smoothable schemes of some length $r$ in a given projective space $\PP^n$ (whenever the Hilbert scheme $\mathrm{Hilb}_r(\PP^n)$ is reducible). However, these are largely unknown. 
\end{remark}

\begin{example}
The assumption of minimal border rank is independent from the vanishing Hessian condition. Let $F = v_0u_0^3 + v_1u_0^2u_1 + v_2u_0u_1^2\in S^4 \CC^{5*}$. So $\mathrm{Hess}(F) = 0$, as the derivatives with respect to $v_i$ are algebraically dependent. One has 
\[
\mathrm{HF}(T/\mathrm{Ann}(F)) = 1 \ \ 5 \ \ 6 \ \ 5 \ \ 1. 
\]
By Remark \ref{limits}, we have $\underline{{\bf r}}(F)\geq 6$ and so $F$ is not of minimal border rank. 

For a cubic example, let $n\geq 8$. Consider $F = x_0x_{n-1}^2 + x_1x_{n-1}x_n + x_2x_n^2 + G(x_3,x_4,\ldots,x_n)\in S^3 \CC^{(n+1)*}$, 
with $G = G(x_3,x_4,\ldots, x_n)\in S^3\CC^{(n-2)*}$ being a general cubic form. By an algebraic relation among the derivatives of $F$ with respect to $x_0, x_1, x_2$, one has $\mathrm{Hess}(F) = 0$. Moreover, we have $\underline{{\bf r}}(F)\geq  \underline{{\bf r}}(G)$, for $G$ is a degeneration of $F$. Whence the border rank of $F$ is generically higher than minimal. 
\end{example}

In the {\it non-minimal} border rank regime, we give the following examples. 

\begin{example}[{\bf Wild + Vanishing Hessian}]\label{wildnonminbr}
Let $H_5 = v_0u_0^4 + v_1u_0^2u_1^2 + v_2u_1^4\in S^5 (\CC^5)^{*}$. Since the partial derivatives with respect to the $v_i$ are algebraically dependent, one has $\mathrm{Hess}(H_5)=0$. A computer algebra 
calculation reveals 
\[
\mathrm{HF}(T/\mathrm{Ann}(H_5)): 1 \ \ 5 \ \ 7 \ \ 7 \ \ 5 \ \ 1. 
\]
By Lemma \ref{borderankGd}, the form $G_5$ satisfies $\underline{{\bf r}}(G_5) = 7$. Since $H_5$ is a degeneration of $G_5$ (i.e. some of the $v_i$ appearing in $G_5$ are sent to zero), $\underline{{\bf r}}(H_5) \leq \underline{{\bf r}}(G_5) = 7$. On the other hand, since the values of the Hilbert function of $\mathrm{Ann}(F)$ give a lower bound for $\underline{{\bf r}}(F)$ (see Remark \ref{limits}), one derives the equality $\underline{{\bf r}}(H_5)=7$. 

As recalled in Remark \ref{limits}, the values of the Hilbert function of $\mathrm{Ann}(F)$ give a lower bound for the cactus rank $\mathrm{cr}(F)$ as well. Therefore $\mathrm{cr}(H_5)\geq 7$. We now show that $\mathrm{cr}(H_5) > 7 = \underline{{\bf r}}(H_5)$, thus proving the wildness of $H_5$. 

For the sake of contradiction, assume $\mathrm{cr}(H_5) = 7$ and let $\mathcal J$ be the saturated ideal of a zero-dimensional scheme of length $7$ such that $\mathcal J\subset \mathrm{Ann}(H_5)$. 

Since $J\subset \mathrm{Ann}(H_5)$, we have $\mathrm{HF}(T/\mathcal J, 2)\geq \mathrm{HF}(T/\mathrm{Ann}(H_5), 2)= 7$ and
$\mathrm{HF}(T/\mathcal J, 3)\geq \mathrm{HF}(T/\mathrm{Ann}(H_5), 3)= 7$. Since $\mathcal J$ is the ideal of a zero-dimensional scheme of length $7$,
by the stabilization of its Hilbert function, we thus conclude $\mathrm{HF}(T/\mathcal J, 2)= \mathrm{HF}(T/\mathcal J, 3) = 7$. Therefore we have the equalities $\mathcal J_2 = \mathrm{Ann}(H_5)_2$ and $\mathcal J_3 = \mathrm{Ann}(H_5)_3$, and hence $\mathcal I = \mathrm{Ann}(H_5)_{\leq 3}\subset \mathcal J$.

Using a computer algebra system, one checks that the saturation of $\mathcal I$ satisfies the equality $\mathcal I^{sat} = (x_0, x_1,x_2)$, where $x_i$ is the dual form to $v_i$. Since $\mathcal I^{sat}\subset \mathcal J^{sat} = \mathcal J$, in particular the ideal $\mathcal J$ contains a linear form. As in the very last part of the proof of Theorem \ref{hessvswild}, this leads to a contradiction. In conclusion, $\mathrm{sr}(H_5)\geq \mathrm{cr}(H_5) > 7=\underline{{\bf r}}(H_5)$ and $H_5$ is wild.
\end{example}

\begin{example}[{\bf Wild + Non-vanishing Hessian}]\label{wildwithnozerohess}
Let $F=v_0u_0^3u_1 + v_1u_0u_1^3+v_0^3v_1^2\in S^5\CC^{4*}$. This satisfies $\mathrm{Hess}(F)\neq 0$. Let $A = T/\mathrm{Ann}(F)$ and let  $\lbrace \alpha_i^{(2)}\rbrace$ be a basis of $A_2$. Then the second Hessian $\mathrm{det}(\alpha_i^{(2)}\alpha_j^{(2)} F) = \mathrm{Hess}^2(F) = 0$. This is a classical example due to Ikeda \cite{Ikeda}, and further revisited by Maeno and Watanabe \cite[Example 5.3]{MW}. This example was independently found by Dias and Gondim \cite[Example 3.15]{DG}, along with other interesting families of instances with the property of being wild and yet having non-vanishing Hessian. 

Now, turning to details, a computer algebra system computation reveals: 
\[
\mathrm{HF}(T/\mathrm{Ann}(F)) = 1 \ \ 4 \ \ 10 \ \ 10 \ \ 4 \ \ 1.
\]
Note that $\mathrm{Ann}(F)_2 = 0$. We first show that $\underline{{\bf r}}(F)= 10$. From the Hilbert function values and by Remark \ref{limits}, we have $\underline{{\bf r}}(F)\geq 10$. To see the upper bound, let us divide the monomials of $F$ as follows: 
\[
F = L_5 + v_0^3v_1^2.
\]
Since border rank is subadditive, we have $\underline{{\bf r}}(F)\leq \underline{{\bf r}}(L_5) + \underline{{\bf r}}(v_0^3v_1^2)$. 
A straightforward computation gives $\underline{{\bf r}}(v_0^3v_1^2)=3$. Now, Lemma \ref{borderankGd} below shows that
$\underline{{\bf r}}(G_5) = 7$. Note that $L_5$ is a degeneration of $G_5$ (i.e. some of the $v_i$ appearing in $G_5$ are sent to zero). 
Thus $\underline{{\bf r}}(L_5)\leq 7$. In fact, we see that 
\[
\mathrm{HF}(T/\mathrm{Ann}(L_5)) = 1 \ \ 4 \ \ 7 \ \ 7 \ \ 4 \ \ 1,
\]
and hence $\underline{{\bf r}}(L_5)= 7$. In conclusion, the equality  $\underline{{\bf r}}(F)= 10$ follows. 

For the sake of contradiction, suppose $\mathrm{cr}(F)=10$ and let $\mathcal J\subset \mathrm{Ann}(F)$ be the saturated
ideal of a zero-dimensional scheme of length $10$. Recall that $\mathrm{Ann}(F)_2=0$ and so is $\mathcal J_2$. 
Since $\mathcal J\subset \mathrm{Ann}(F)$, one has $\mathrm{HF}(T/\mathcal J, 3)\geq \mathrm{HF}(T/\mathrm{Ann}(F), 3)=10$. 
By the stabilization property of the Hilbert function, it follows that $\mathrm{HF}(T/\mathcal J, 3)=10$ and so $\mathcal I = \mathrm{Ann}(F)_{\leq 3}\subset \mathcal J$. Therefore $\mathcal I^{sat} \subset \mathcal J^{sat} = \mathcal J$. 

It is a direct calculation to show that $\mathcal I^{sat}\subset T$ contains a quadric and so does $\mathcal J$. This leads to a contradiction as $\mathcal J_2=0$. In conclusion, we must have $\mathrm{sr}(F)\geq \mathrm{cr}(F)>10 = \underline{{\bf r}}(F)$, and hence $F$ is wild. 
\end{example}

\begin{example}[{\bf Non-wild + Vanishing Hessian}]
Let $F= v_0u_0^3 + v_1u_1^3 + v_2(u_0+u_1)^3 \in S^4 (\CC^5)^{*}$. 
Then it is straightforward to check that $\mathrm{Hess}(F) = 0$ and $\underline{{\bf r}}(F) = 6$. Note that $F$ is in the span of 
a scheme $R$ consisting of three $2$-jets and so smoothable. Thus $\mathrm{sr}(F) = 6$.
Therefore $F$ is not wild. 
\end{example}

\begin{proposition}
Keep the notation from Theorem \ref{hessvswild}. Let $F$ be a concise and minimal border rank cubic. If $F$ is not a wild cubic and $R^W$ is reduced, then $F$ is a Fermat
cubic (up to scaling variables) $F = x_0^3+\cdots + x_n^3$, and $\underline{\mathrm{VSP}}(F, n+1)$ is a single point. 
\begin{proof}
Since $F$ is concise, of minimal border rank and it is not wild, the proof of Theorem \ref{hessvswild} yields that $\mathcal{I}^{sat}$ does not contain any linear form. This is equivalent to $\langle R^W \rangle = \mathbb{P}^n$, where $W =  \langle \frac{\partial}{\partial x_i} F\rangle\subset \PP(S^{2} V^{*})$. 
Since $R^{W}$ is a reduced scheme by assumption, we may find a reduced zero-dimensional subscheme $Z\subset R^{W}$ of length $n+1$ such that $\langle Z\rangle = \PP^n$. Up to change of basis, we have:
\[
\left((\mathcal I_Z)_2\right)^{\perp} = \langle x_0^{2},\ldots, x_n^{2}\rangle \subseteq \left((\mathcal{I}^{sat})_2\right)^{\perp} \subseteq W.
\]
Since $\dim W = n+1$, we have $W = \langle x_0^{2},\ldots, x_n^{2}\rangle$ and
\[
\mathcal I = \mathrm{Ann}(F)_{2} = W^{\perp} = \langle x_0^{2},\ldots, x_n^{2}\rangle^{\perp}.
\]
Thus $\mathrm{Ann}(F)_{3} = \langle x_0^3, \ldots, x_n^3\rangle^{\perp}$. Then, since every cubic monomial divisible by two distinct 
variables is in $\mathrm{Ann}(F)_{3}$, $F$ is a Fermat cubic up to the action of a diagonal matrix. Now, the Hilbert function of $\mathcal I$ is:
\[
\mathrm{HF}(T/\mathcal I): 1 \ \ (n+1) \ \ (n+1) \ \ (n+1) \ \ \cdots
\]
For $k=0,1,2$, the dimension of $\mathcal I_k$ is clear from definitions. To see the dimension of $\mathcal I_{k}$ for $k\geq 3$, note that $y_0^{k},\ldots, y_n^{k}\notin \mathcal I_{k}$ and they are the only missing monomials.  Let $\mathcal J\in \underline{\mathrm{VSP}}(F, n+1)$. Then its Hilbert function is
\[
\mathrm{HF}(T/\mathcal J): 1 \ \ (n+1) \ \ (n+1)  \ \ (n+1) \ \ \cdots, 
\]
because $\mathcal J\in \mathrm{Slip}_{n+1,\PP^n}$. Since $\mathcal J\subset \mathrm{Ann}(F)$, we have the equality $\mathcal I = \mathcal J$. \end{proof}
\end{proposition}

Repeating part of the proof above, one shows: 

\begin{proposition}\label{uniquenesswhensaturated}
Keep the notation from Theorem \ref{hessvswild}. Let $F$ be a concise and minimal border rank cubic. If $F$ is not a wild cubic and $\mathcal I =\mathrm{Ann}(F)_2$ is saturated of degree $n+1$, then $\underline{\mathrm{VSP}}(F, n+1) = \lbrace \mathcal I\rbrace$.  
\end{proposition}

\begin{example}
Let $F = x_0^2x_1\in S^3 \CC^{2*}$. In this case, $\mathcal I = \mathrm{Ann}(F)_2 = \langle y_1^2\rangle\subset T$ is the ideal of a $2$-jet on a $\PP^1$. Proposition \ref{uniquenesswhensaturated} gives $\underline{\mathrm{VSP}}(F, 2) = \lbrace \mathcal I\rbrace$. 
\end{example}

\begin{example}
Let $F_{\mathrm{tg}} = x_1(x_0^2+x_1x_2) \in S^3 \CC^{3*}$ (a conic with a tangent line) or $F_{\mathrm{cusp}} = x_1^2x_2-x_0^3\in S^3 \CC^{3*}$ (a cuspidal cubic).  In both cases, $\mathcal I = \mathrm{Ann}(F)_2$ is saturated of degree $3$. For $F_{\mathrm{tg}}$, the scheme defined by $\mathcal I$ is the $2$-fat point (of length $3$) in $\PP^2$; for $F_{\mathrm{cusp}}$, the scheme defined by $\mathcal I$ is the union of a simple point and a $2$-jet. 

In both cases, Proposition \ref{uniquenesswhensaturated} yields $\underline{\mathrm{VSP}}(F, 3) = \lbrace \mathcal I\rbrace$. 
\end{example}

\section{Smoothable algebras, structure tensors and wild forms}\label{cubicssmoothablealgebras}

In this section, for the ease of notation, we regard $F\in S^d V$ as forms (instead of using duals). A tensor in $V^{\otimes d}$ is called a {\it $d$-way tensor}. 

Recall that a form $F\in S^d V$ is a symmetric tensor $T_{F}\in V^{\otimes d}$ (the identification is defined in characteristic zero). Here, we bring tensors into the picture in order to establish the converse to Theorem \ref{hessvswild}. With this aim at hand, we start making a detour through smoothable algebras and structure tensors. We shall demonstrate Theorem \ref{maincubics}, providing  a classification of wild forms of minimal border rank, thus complementing Theorem \ref{hessvswild}. 

In the rest, a finite-dimensional $\CC$-algebra $A$ is a finite-dimensional $\CC$-vector space with an associative unital multiplication. 

\begin{definition}[{\bf Smoothable algebras}]
Let $A$ be a (commutative) $\CC$-algebra of dimension zero as a ring and $n+1$ as a $\CC$-vector space. The algebra $A$ is said to be {\it smoothable} if it is a degeneration of the standard algebra structure on $\CC^{n+1}$. Equivalently, in scheme theory terminology, $\Spec(A)$ is a smoothable zero-dimensional scheme. The dimension of $A$ as a $\CC$-vector space is the length of $\Spec(A)$. 
\end{definition}

Since we deal with smoothable algebras $A$, henceforth we restrict our discussion to commutative algebras. For details about the subsequent
material, we refer to \cite[Lecture 8]{Geramita} or \cite[\S 3 and \S 16]{Lam}. 

A finite-dimensional algebra $A$ is an Artinian ring, so $A$ has finitely many maximal ideals. Let $A$ be a local finite-dimensional $\CC$-algebra (i.e., $A$ is a local ring) with maximal ideal $\mathfrak{m}$. Its {\it socle} is the set of all ring elements $a\in A$ such that $a\in (0:\mathfrak{m})$, i.e.~$a\cdot \mathfrak{m} = 0$. 

\begin{definition}[{\bf Gorenstein algebras}]\label{gorenstein}

Let $A$ be a local finite-dimensional $\CC$-algebra. The algebra $A$ is (local) {\it Gorenstein} if one of  the following two equivalent conditions hold true: 
\begin{enumerate}

\item[(i)] its socle is a one-dimensional $\CC$-vector space; 

\item[(ii)] there exists a perfect pairing $p: A\times A \rightarrow \CC$, defined by $p(a,b)=e(ab)$ for a linear form $e: A\rightarrow \CC$; see e.g. \cite[Theorem 3.15]{Lam}.

\end{enumerate}

A finite-dimensional algebra $A$ is said to be {\it Gorenstein} if every localization $A_{\mathfrak{m}}$ at a maximal ideal $\mathfrak{m}\subset A$
is a local Gorenstein algebra. 
\end{definition}

\begin{definition}[{\bf The $d$-way structure tensor of an algebra}]
Let $A$ be a unital, commutative finite-dimensional $\CC$-algebra. The multiplication map $m_{A}: A \times \ldots \times A \rightarrow A$ given by $m_A(a_1,a_2,\ldots,a_{d-1}) = a_1 a_2 \ldots a_{d-1}$ is multilinear and symmetric. Therefore, one may regard it as a partially symmetric tensor $T_A\in \mathrm{S}^{d-1} A^* \otimes A \subseteq A^{*}\otimes  \ldots \otimes A^* \otimes A$. The tensor $T_A$ is the {\it $d$-way structure tensor} of $A$. In particular, when $d=3$, this is the usual {\it structure tensor} of $A$ \cite{BL}. 

Let $V$ be a finite-dimensional $\CC$-vector space. An (abstract) tensor $T\in V^{*}\otimes \ldots \otimes V^* \otimes V$ is the $d$-way structure tensor of an algebra $A$ if $T$ is isomorphic to $T_A$, i.e. if there exist $d$ linear isomorphisms $\phi_1,\ldots, \phi_{d}$ such that $(\phi_1\otimes \phi_2\otimes \ldots \otimes \phi_d)(T) = T_A$, where $\phi_i$ is a  linear isomorphism among the $i$-th factors.

We summarize the previously known results about ($3$-way) structure tensors.

\end{definition}

\begin{definition}[{\bf 1-generic tensors}]
Let $T\in V\otimes V\otimes V$ be a tensor. Then $T$ is $1$-{\it generic} if its contraction in {\it every} factor $T(V^{*})$ is a linear space
containing a full-rank matrix $M\in T(V^{*})\subset V\otimes V$. If $T$ is symmetric, it is enough to require that a contraction in only one factor has the desired property.
\end{definition}

\begin{remark}\label{partialonegeneric}
Let $V_1, V_2$ and $V_3$ be three finite-dimensional $\CC$-vector spaces. Suppose $T\in V_1\otimes V_2\otimes V_3$ 
is a structure tensor of some algebra $A$. Then the linear spaces $T(V_1^{*})\subset V_2\otimes V_3$ and $T(V_2^{*})\subset V_1\otimes V_3$ 
contain full rank elements. ($T$ is said to be {\it binding} \cite[Lemma 3.5]{BL}, or $1_{V_1}$- and $1_{V_2}$-generic \cite[\S 1]{LM}.)
\end{remark}

Bl\"{a}ser and Lysikov characterized $3$-way structure tensors of algebras of minimal border rank 
as the ones coming from smoothable algebras \cite[Theorem 3.2, Corollary 3.3]{BL}: 

\begin{theorem}\label{blthm}
A $3$-way tensor has minimal border rank and is $1_{V_1}$- and $1_{V_2}$-generic if and only if it is isomorphic to a structure tensor $T_A$ of a smoothable algebra $A$. 
\end{theorem}

Now we establish a similar result for $d$-way structure tensors, for arbitrary $d$, in Proposition~\ref{prop:Gor} through the following lemma. 

\begin{lemma}\label{lem:Hes1gen}
Let $F\in S^d V$ be a concise form regarded as a symmetric concise tensor  $T_F\in V^{\otimes d}$. The following statements are equivalent:
\begin{enumerate}
\item[(i)] $\mathrm{Hess}(F)\neq 0$;
\item[(ii)] there exists $\ell \in V^*$ such that the contraction $T_F(\ell^{\otimes d-2})$ is a full-rank symmetric matrix.
\end{enumerate}
In particular, a cubic $F\in S^3 V$ has non-vanishing Hessian if and only if $T_F$ is a $1$-generic symmetric tensor. 
\begin{proof}
Let $A = T/\mathrm{Ann}(F)$ be the Artinian Gorenstein $\CC$-algebra of $F$. The condition $\mathrm{Hess}(F)\neq 0$ is equivalent to the fact that there exists a linear form $\ell\in A_1$ such that $\phi_{\ell^{d-2}}: A_1\rightarrow A_{d-1}$ is an isomorphism (the map here is multiplication by $\ell^{d-2}$); see the proof of \cite[Theorem 3.1]{MW} or \cite[Theorem 7.2.20]{russo}. The map $\phi_{\ell^{d-2}}$ is an isomorphism if and only if the quadric $Q=\ell^{d-2}(F)\in S^2 V$ is non-degenerate.\end{proof}
\end{lemma}
We will use the previous equivalence for concise forms of minimal border rank having non-vanishing Hessian. 
Before proceeding further, notice that if $F\in S^d V$ is a concise form of minimal border rank, then its corresponding tensor $T_F$ has minimal border (tensor) rank.

\begin{proposition}\label{prop:Gor}
Let $V$ be an $(n+1)$-dimensional $\CC$-vector space.
Let $F\in S^d V$ be a concise form of minimal border rank, and such that $\mathrm{Hess}(F)\neq 0$. 
Then $T_F$ is the $d$-way structure tensor of a smoothable Gorenstein algebra $A$. 

Moreover, if $T$ is the structure tensor of an $(n+1)$-dimensional smoothable algebra $A$, then $T$ is (isomorphic to) a symmetric tensor if and only if $A$ is a Gorenstein algebra.
\end{proposition}
\begin{proof}
By Lemma~\ref{lem:Hes1gen}, we can find $\ell \in V^*$ such that the contraction $T_F(\ell^{\otimes d-2})$ is a full-rank symmetric matrix.  We first fix a linear basis $v_1 = \ell, v_2, \ldots,v_{n+1}$ of $V^*$. We may regard $T_F(\ell^{\otimes d-2})$ as a bilinear map $Q \colon V^* \times V^* \to \CC$. From here, we fix another linear basis $v_1',v_2',\ldots,v_{n+1}'$ of $V^*$ such that $Q(v_i,v_j') = \delta_{ij}$. Now, we define $T_F^{i_1 i_2 \ldots i_d}$ to be $T_F(v_{i_1}\otimes v_{i_2}\otimes \ldots \otimes v_{i_{d-1}} \otimes v_{i_d}')$. By definition, $T_F^{i_1 i_2 \ldots i_d}$ is symmetric with respect to the indices $(i_1 i_2 \ldots i_{d-1})$ and $T_F^{11\ldots1kl} = \delta_{kl}$.

We define the matrices $M_i := T_F^{1,1,\ldots,1,\bullet,i,\bullet}$ for $1 \leq i \leq n+1$. We will prove the following identities that are the {\it defining} properties 
of the $d$-way structure tensor of an algebra. 
\begin{itemize}
\item {\bf Commutativity}: 
\[ 
M_i M_j = M_j M_i.
\]
\item {\bf Closed under composition}: 
\[
M_i M_j = \sum_k T_F^{11 \ldots 1 i j k} M_k.
\]
\item {\bf Structure property}: 
\[
M_{i_1} M_{i_2} \ldots M_{i_{d-1}} = \sum_{j = 1}^{n+1} T_F^{i_1 i_2 \ldots i_{d-1} j} M_j .
\]
\end{itemize}

\noindent {\bf Commutativity}. Regard $T_F \in (V^{\otimes d-2}) \otimes V \otimes V$ as a $3$-way tensor. Since $F$ is concise of minimal border rank, the border rank of $T_F$ as a $3$-way tensor is also $n+1 = \dim V$. The commutativity of $\{M_i\}$ then follows from \cite[Lemma 2.6]{LM}.\\

\noindent {\bf Closed under composition}. Regard $T_F \in (V^{\otimes d-2}) \otimes V \otimes V$ as a $3$-way tensor and apply \cite[Proposition 2.10]{LM}.\\

\noindent {\bf Structure property}. Regarding once again $T_F$ as a $3$-way tensor, $T_F$ satisfies {\it Strassen's commutator equations} (see \cite{Str} or \cite[\S 2.1]{LM}),
which in the given coordinates are: 
\[
\sum_{k} T_F^{i_1 i_2 \ldots i_{d-2} j_1 k} T_F^{i_1' i_2' \ldots i_{d-2}' k j_2} = \sum_{k} T_F^{i_1' i_2' \ldots i_{d-2}' j_1 k} T_F^{i_1 i_2 \ldots i_{d-2} k j_2}.
\]
Using the symmetry of $T_F$ in the first $(d-1)$ indices, the equality $T_F^{11\ldots1kl} = \delta_{kl}$, and Strassen's equations, we have the following identities:
\begin{align*}
\sum_k  T_F^{1 j_1 j_2 \ldots j_{d-2} k} T_F^{1 1 \ldots 1 k j_{d-1} j_d} & = \sum_k T_F^{j_1 j_2 \ldots j_{d-2} 1 k} T_F^{1 1 \ldots 1 j_{d-1} k j_d} \\
                                                                          & = \sum_k T_F^{1 1 \ldots 1 j_{d-1} 1 k} T_F^{j_1 j_2 \ldots j_{d-2} k j_d} \\
                                                                          & = \sum_k T_F^{1 1 \ldots 1 j_{d-1} k} T_F^{j_1 j_2 \ldots j_{d-2} k j_d} \\
                                                                          & = T_F^{j_1 j_2 \ldots j_{d-1} j_d}.
\end{align*}
Repeatedly applying the above identities, we have the following equalities:
\begin{align*}
    (M_{i_1} M_{i_2} \ldots M_{i_{d-2}})_{ab} & = \sum_{k_1,\ldots,k_{d-3}} T_F^{11\ldots1 a i_1 k_1} T_F^{11\ldots1 k_1 i_2 k_2} T_F^{11 \ldots 1 k_2 i_3 k_3} \cdots T_F^{11 \ldots 1 k_{d-3} i_{d-2} b} \\
                                              & = \sum_{k_2,\ldots,k_{d-3}} \left(\sum_{k_1} T_F^{11\ldots1 a i_1 k_1} T_F^{11 \ldots 1 k_1 i_1 k_2}\right) T_F^{11 \ldots 1 k_2 i_3 k_3} \cdots T_F^{11\ldots 1 k_{d-3} i_{d-2} b} \\
                                              & = \sum_{k_3,\ldots,k_{d-3}} \left(\sum_{k_2} T_F^{11\ldots 1 a i_1 i_2 k_2} T_F^{11 \ldots 1 k_2 i_3 k_3}\right) \cdots T_F^{11\ldots 1 k_{d-3} i_{d-2} b} \\ 
                                              & = \sum_{k_{d-3}} T_F^{1 a i_1 i_2 \ldots i_{d-3} k_{d-3}} T_F^{11 \ldots 1 k_{d-3} i_{d-2} b} \\  & = T_F^{a i_1 i_2 \ldots i_{d-2} b} \\
                                              & = T_F^{i_1 i_2 \ldots i_{d-2} a b} \\
                                              & = \sum_{k} T_F^{i_1 i_2 \ldots i_{d-2} 1 k} T_F^{11 \ldots 1 a k b} \\  
                                              & = \sum_k T_F^{1 i_1 i_2 \ldots i_{d-2} k} T_F^{11 \ldots 1 a k b} \\
                                              & = \left(\sum_k T_F^{1 i_1 i_2 \ldots i_{d-2} k} M_k\right)_{ab}.
\end{align*}
Therefore, using the last equality, we find
\begin{align*}
\left(M_{i_1} M_{i_2} \cdots M_{i_{d-2}}\right)\cdot M_{i_{d-1}} & = \left(\sum_k T_F^{1 i_1 i_2 \ldots i_{d-2} k} M_k\right)\cdot M_{i_{d-1}} \\
                                   & = \sum_j \sum_k T_F^{1 i_1 i_2 \ldots i_{d-2} k} T_F^{11\ldots1 k i_{d-1} j} M_j \\
                                   & = \sum_j T_F^{i_1 i_2 \ldots i_{d-2} i_{d-1} j} M_j,
\end{align*}
where in the second line we use the identity in the {\bf Closed under composition} property. This establishes the {\bf Structure property}. \\

\noindent Let $\mathrm{M}_n(\CC)$ be the algebra of $n\times n$ complex matrices. Consider the map 
\[
\phi \colon \CC [x_1,\ldots,x_{n+1}] \to \mathrm{M}_n(\CC),
\] 
defined by $\phi(x_i) = M_i$. Because of the three properties above of the matrices $\{ M_i \}$, this is a ring map whose image is an $(n+1)$-dimensional vector space. 

Define the algebra $A := \CC[x_1,\ldots,x_{n+1}] / \ker \phi$. Then by definition and by the three properties above, $T_F$ is the $d$-way structure tensor of $A$. As $T_F$ is of minimal border rank, so is the $3$-way structure tensor of the algebra $A$. Thus $A$ is smoothable by Theorem \ref{blthm}. This finishes the proof of the first part of the proposition (the Gorenstein property will be proven as corollary of the second part). 

For the second part, assume $T$ is a $d$-way structure tensor of an $(n+1)$-dimensional smoothable algebra $A$. Let $\mathrm{u} \in A$ be the identity element in $A$. Then $T(\mathrm{u}^{\otimes d-2}) \colon A \to A$ is the identity matrix. Hence contracting $T$ with respect to the $(d-2)$ left-most factors using only tensors of rank $1$ produces a full-rank element. 

Suppose $T$ is symmetric. Then contracting $T$ with respect to the $(d-2)$ right-most factors using only tensors of rank $1$ can produce a full-rank element. Regarding $T \in (A^*)^{\otimes d-1} \otimes A$, we may find $a_3 \otimes a_4 \otimes \ldots \otimes a_{d-1} \otimes \alpha^k \in A^{\otimes d-3} \otimes A^*$ such that $T(a_3 \otimes \ldots \otimes a_{d-1} \otimes \alpha^k) \colon A \times A \to \CC$ is a perfect pairing. Define $e \colon A \to \CC$ to be $e(a) = \alpha^k(a a_3 a_4 \cdots a_{d-1})$. Then $p = T(a_3 \otimes \ldots \otimes a_{d-1} \otimes \alpha^k)$ and the above $e$ give us the required properties in Definition~\ref{gorenstein}(ii). Hence $A$ is Gorenstein.

Suppose $A$ is a Gorenstein algebra. Then, by Definition \ref{gorenstein}(ii), there exists a perfect pairing $p:A\times A\rightarrow\CC$, where $p(a , b)=e(a b)$ for some linear form $e: A\rightarrow \CC$. 

Let us fix a nonzero vector $\alpha^k\in A^{*}$. Since $T\in (A^{*})^{\otimes d-1} \otimes A$, its coordinates are functions in the dual 
vector space, i.e. $a_{i_1}\otimes \ldots \otimes a_{i_{d-1}}\otimes \alpha^k\in A^{\otimes d-1}  \otimes A^{*}$ is a function on $T$. The $a_{i_1}\otimes \ldots \otimes a_{i_{d-1}}\otimes \alpha^k$-coordinate of $T$ is $\alpha^k(a_{i_1} \cdots a_{i_{d-1}})$ because $T$ is the $d$-way  structure tensor of $A$. Recall we have a perfect pairing $p: A\times A\rightarrow \CC$ inducing an isomorphism (which we denote with the same name) $p: A\rightarrow A^{*}$. Letting $c = p^{-1}(\alpha^k)$, by definition one has $\alpha^k(a) = e(c a)$ for all $a\in A$. Now, we choose $\alpha^k$ so that $\alpha^k(a_{i_1} \cdots a_{i_{d-1}}) = e(a_k a_{i_1} \cdots a_{i_{d-1}})$ for every $k$, i.e.~$\alpha^k=p(a_k)$. This identifies $T$ as a symmetric tensor in $(A^{*})^{\otimes d}$. 
\end{proof}
\begin{theorem}\label{maincubics}
Let $V$ be an $(n+1)$-dimensional $\CC$-vector space. Let $F\in S^d V$ be a concise form of minimal border rank. Then:
\[
\mathrm{cr}(F) = n+1 \Longleftrightarrow \mathrm{Hess}(F)\neq 0  \Longleftrightarrow  \mathrm{sr}(F) = n+1.
\] 
\end{theorem}
\begin{proof}
By Theorem \ref{hessvswild}, if $\mathrm{cr}(F)=n+1$ then $\mathrm{Hess}(F)\neq 0$. We now show: 
\[
\mathrm{Hess}(F)\neq 0 \Longrightarrow \mathrm{sr}(F)=n+1. 
\]
Note that this is enough to prove the statement, as $\mathrm{sr}(F) = n+1$ implies $\mathrm{cr}(F) = n+1$.

Suppose $\mathrm{Hess}(F)\neq 0$. By Proposition~\ref{prop:Gor}, the symmetric tensor $T_F$ corresponding  to $F$ is the $d$-way structure tensor of a smoothable Gorenstein algebra $A$. Let $p(a,b)=e(ab)$ be the perfect pairing on $A$, inducing an isomorphism $p:A\rightarrow A^*$. As in the proof of Proposition \ref{prop:Gor}, we may realize $T_F$ as a symmetric tensor in $(A^*)^{\otimes d}$. We fix a linear basis $a_1=1, a_2,\dots, a_{n+1}$ of $A$. The $(i_1,\ldots,i_d)$-th entry of the tensor $T_F$ equals $e(a_{i_1}\ldots a_{i_d})$.

We embed the affine scheme $\Spec(A)$ in $\CC^n$ via the surjective ring map: 
\[
\phi_A: \CC[x_2,\dots,x_{n+1}]\rightarrow A,
\]
where $x_i\mapsto a_i$. Next, we embed $\CC^n$ in $\PP^n=\PP(A^*)$ and apply the $d$-th Veronese embedding to $\PP(S^d A^*)$.  Let $L\subset \PP(S^d A^*)$ be the affine subspace given by the complement of the zero locus of the function $1\cdot 1 \cdots 1\in S^d A$. Write $L = \Spec (\CC[y_{i_1 i_2 \ldots i_d}])$, where $\{i_1,i_2,\ldots,i_d\}$ is a cardinality $d$ multisubset of $\{1,\dots,n+1\}$ distinct from $\{1,1,\ldots,1\}$. The embedding $\nu_d(\Spec(A))$ in $\nu_d(\CC^n)\subset L\subset \PP(S^d A^*)$ is defined by the ring map $y_{i_1 i_2 \ldots i_d}\mapsto a_{i_1}a_{i_2}\ldots a_{i_d}\in A$. 

We claim that $T_F$ belongs to the linear span of $\nu_d(\Spec(A))$. To see this, consider any linear form $h =\sum \lambda_{i_1 i_2 \ldots i_d} y_{i_1 i_2 \ldots i_d}$ vanishing on $\nu_d(\Spec (A))$. This means that $\sum \lambda_{i_1 i_2 \ldots i_d} a_{i_1} a_{i_2} \ldots a_{i_d}=0\in A$. Applying the linear form $e$ we obtain: \[0= e(0)= e\left(\sum \lambda_{i_1 i_2 \ldots i_d} a_{i_1} a_{i_2} \ldots a_{i_d}\right)=\sum \lambda_{i_1 i_2 \ldots i_d} e(a_{i_1} a_{i_2} \ldots   a_{i_d})= h(T_F).\]
Hence the linear form $h$ vanishes on $T_F$. This shows that $T_F$ belongs to the linear span of $\nu_d(\Spec(A))$, which implies $\mathrm{sr}(F) = n+1$, thus finishing the proof.\end{proof}

\section{The limiting scheme}\label{seclimsch}

We start with  the definition of limiting scheme: 

\begin{definition}\label{borderdec}
Let $F\in S^d V^{*}$ be a form. Suppose we are given a border rank decomposition for $F$, i.e. 
\begin{equation}\label{brdec}
F = \lim_{t\rightarrow 0}\frac{1}{t^s}\left(L_1(t)^d + \ldots  +L_{\underline{{\bf r}}(F)}(t)^d\right), \ s\geq 0, 
\end{equation}
where $L_i(t)$ are linear forms. The reduced zero-dimensional scheme whose (closed) points are the $L_j(t)$ is denoted $R(t)$. 
 For each $t\neq 0$, the radical ideal defining $R(t)$ is denoted $\mathcal I_{R(t)}$. The flat limit $Z = \lim_{t\rightarrow 0} R(t)$ is called the {\it limiting scheme} of \eqref{brdec}. Note that $Z$ and $R(t)$ ($t\neq 0$) have the same Hilbert polynomial; see e.g. \cite[Theorem III.9.9]{h}. 
\end{definition}

We have the following corollary from the proof of \cite[Proposition 2.6]{bb15}:

\begin{corollary} \label{tangentcone1}
Keep the assumptions from Proposition~\ref{tangentcones}. Assume 
\[
F \in \langle \PP \widehat{T}_{z_1}, \ldots, \PP \widehat{T}_{z_r}\rangle\subset \sigma_r(X).  
\]
Then we can find a border rank decomposition for $F$ whose limiting scheme is the smooth scheme supported at the $r$ points $\{z_1,\ldots,z_r\}$.
\end{corollary}

\begin{proof}
Assume $\DS F = \sum_{i=1}^{r} v_{z_i}$ where $[v_{z_i}] \in \PP \widehat{T}_{z_i}$. Let $\hat{z}_i \in \CC^{N_d+1}$ such that $[\hat{z}_i] = z_i$ and $\DS \sum_i \hat{z}_i = 0$. We can find curves $\hat{z}_1 (t),\ldots,\hat{z}_r (t)$ in the affine cone $\widehat{X}$ over $X$ such that $\hat{z}_i (0) = \hat{z}_i$ and $\frac{\mathrm{d} \hat{z}_i}{\mathrm{d} t}(0) = v_{z_i}$. Then we have the following border rank decomposition for $F$:
\[
F = \lim_{t \rightarrow 0} \frac{1}{t} \sum_{i=1}^r \hat{z}_i (t).
\]
Since $\{z_1,\ldots,z_r\}$ are $r$ distinct points on $X$, the limiting scheme corresponding to the border rank decomposition above is the smooth scheme supported at the $r$ points $\{z_1,\ldots,z_r\}$.
\end{proof}

The next result is a consequence of Buczy\'nska-Buczy\'nski's theory: 

\begin{theorem}\label{satidealsvsp}
The saturation of any ideal in $\underline{\mathrm{VSP}}(F, \underline{{\bf r}}(F))$ coincides with the ideal of a limiting scheme of a border
rank decomposition.

\begin{proof}
Any ideal $\mathcal J\in \underline{\mathrm{VSP}}(F,\underline{{\bf r}}(F))$ comes from some border rank decomposition \eqref{brdec}; see the proof of \cite[Theorem 3.15]{bb19}. Such a border rank decomposition \eqref{brdec} determines a family of zero-dimensional schemes $R(t)$, each of length $\underline{{\bf r}}(F)$, such that their ideals $\mathcal I_{R(t)}$ have 
the generic Hilbert function and $\mathcal J =  \lim_{t \rightarrow 0} \mathcal{I}_{R(t)}$. 

By definition of limits, we have
\[
\mathcal J = \lim_{t \rightarrow 0} \mathcal{I}_{R(t)} \subset \mathcal{I}_{\lim_{t \rightarrow 0} R(t)}. 
\] 

Let $Z = \lim_{t\rightarrow 0} R(t)$ be the limiting scheme of the given border rank decomposition \eqref{brdec}. Since $Z$ has length $\underline{{\bf r}}(F)$, its ideal has the same degree (or Hilbert polynomial) as $\mathcal J$. Since $\mathcal J\subset \mathcal{I}_Z$, their saturations coincide. Since $\mathcal{I}_Z$ is saturated by definition, $\mathcal J^{sat} = \mathcal{I}_Z$. 
\end{proof}
\end{theorem}

\begin{theorem}\label{correspondence}
Suppose $F\in S^d V^{*}$ is concise and of minimal border rank $n+1$. Then every border rank decomposition of $F$ determines an ideal in $\underline{\mathrm{VSP}}(F, n+1)$.
\begin{proof}
Given any border rank decomposition $F = \lim_{t\rightarrow 0} H(t) $, where each $H(t) = \frac{1}{t^s}\left(L_1(t)^d + \ldots  +L_{n+1}(t)^d\right)$ has rank $n+1$,
one has that $\mathcal I_{R(t)}\subset \mathrm{Ann}(H(t))$ ($t\neq 0$) by the classical Apolarity lemma \cite[Lemma 1.15]{ik}. Since $F$ is concise, we can find $t \neq 0$ such that $H(t)$ is concise. In this case, $R(t)$ consists of $n+1$ linearly independent points. Therefore $\mathcal I_{R(t)}$ has the generic Hilbert function of $n+1$ points in $\PP^n$. Hence $\mathcal J:= \lim_{t\rightarrow 0} \mathcal I_{R(t)}\in \underline{\mathrm{VSP}}(F, n+1)$. 
\end{proof}
\end{theorem}

\section{Wild forms of higher degree and their $\underline{\mathrm{VSP}}$}\label{higherdeg}

Let $d\geq 3$ and define the following infinite series of concise forms of degree $d$:
\[
G_d = \sum_{i=0}^{d-1} v_i u_0^i u_1^{d-1-i}. 
\]

For $d=3$, this coincides up to change of variables with the wild cubic form found in \cite[\S 4]{bb15}. This infinite series has
a classical geometric significance: they are the equations of the dual hypersurfaces of the rational ruled surfaces $\PP(\mathcal O_{\PP^1}(1)\oplus \mathcal O_{\PP^1}(d-1))$ embedded with the tautological bundle in $\PP^{d+1}$; see \cite[Theorem 3.1]{GRS20} and \cite[Chapter 3, Example 3.6]{GKZ}. (In the suggestive classical terminology, these rational surfaces are called {\it rational normal scrolls with line directrix}.)

Note that, for every $d\geq 3$, the partial derivatives of $G_d$ with respect to the variables $v_i$ are algebraically dependent: their relations coincide with the equations of the usual Veronese embedding of degree $d-1$ of $\PP^1$ in $\PP^{d-1}$. Thus $\mathrm{Hess}(G_d) = 0$ for $d\geq 3$.
Let $\mathrm{Ann}(G_d)\subset T = \CC[x_0,\ldots, x_{d-1}, y_0, y_1]$ be its annihilator, where $x_i$ is dual to $v_i$ and $y_j$ is dual to $u_j$.

\begin{lemma}\label{borderankGd}
For every $d\geq 3$, the form $G_d$ has minimal border rank $d+2$. 
\begin{proof}
Let $\ell_0^d,\ldots, \ell_{d}^d$ be $d+1$ pairwise distinct linear forms in $u_0, u_1$. They may 
be viewed as $d+1$ distinct points on the degree $d$ rational normal curve $\nu_d(\PP^1)\subset \PP^d$. 
They are linearly independent. (This is well-known and can be explicitly checked, for instance, by calculating the corresponding Wronskian matrix at the origin and show it is full rank.)
Any other form $\ell_{d+1}^d$ is linearly dependent to those above, because $\dim S^d\langle u_0, u_1\rangle = d+1$. 

Up to change of bases one has $G_d = \sum_{i=0}^{d-1}v_i\ell_{i}^{d-1}$. Let $\PP\widehat{T}_{\ell_i^{d}}$ denote the affine cone of the Zariski
tangent space to $\nu_d(\PP^1)$ at $\ell_i^{d}$. Therefore 
\[
G_d\in \left\langle \PP\widehat{T}_{\ell_i^{d}}\ , 0\leq i\leq d-1\right\rangle. 
\]
Since $\ell_0^d,\ldots, \ell_{d}^d, \ell_{d+1}^d$ are linearly dependent, Proposition \ref{tangentcones} yields the inequality $\underline{{\bf r}}(G_d)\leq d+2$.
Since $G_d$ is concise, $\underline{{\bf r}}(G_d)\geq d+2$. Thus equality holds. 
\end{proof}
\end{lemma}

\begin{corollary}\label{infiniteseriesdeg}
The forms $G_d$ are wild. 
\begin{proof}
By Lemma \ref{borderankGd}, $G_d$ has minimal 
border rank. Moreover, as noticed above, $\mathrm{Hess}(G_d) = 0$. Theorem \ref{hessvswild} shows that the degree $d$ forms $G_d$ are wild. 
\end{proof}
\end{corollary}

\begin{remark}
For every $d$, one has 
\[
d+2 = \underline{{\bf r}}(G_d) < \mathrm{cr}(G_d)\leq \mathrm{sr}(G_d)\leq 2d,
\] 
as every $2$-jet is smoothable and so is their union. 
\end{remark}

This corollary complements \cite[Theorem 1.3]{bb15}, as follows: 

\begin{theorem}
For every $d\geq 3$, there exist wild forms of degree $d$. 
\end{theorem}

\begin{proposition}\label{structureidealVSPGd}
Let $\mathcal J\subset T$ be an ideal. Then $\mathcal J\in \underline{\mathrm{VSP}}(G_d, d+2)$ if and only if $\mathcal J = \mathrm{Ann}(G_d)_{\leq d-1}+ \mathcal Q$, 
where $\mathcal Q$ is the principal ideal generated by a form $q\in S^{d+2}\langle y_0,y_1\rangle$. 
\begin{proof}
Let $\mathcal J\in  \underline{\mathrm{VSP}}(G_d, d+2)$ and let $\mathcal I =  \mathrm{Ann}(G_d)_{\leq d-1}$. 

Then $\mathcal J\in \mathrm{Slip}_{d+2,\PP^{d+1}}$ and, by definition, the Hilbert function of $\mathcal J$ is the generic Hilbert function of $d+2$ points in $\PP^{d+1}$. Since $\mathcal J\subset \mathrm{Ann}(G_d)$, it follows that $\mathcal J\supset \mathcal I$, as they must coincide up to degree $d-1$. 

Now, consider the Hilbert function of $\mathcal I$. The following relations hold in $T / \mathcal{I}$:
\begin{align*}
x_i x_j \equiv 0 \ \ (\text{mod } \mathcal{I}) \quad & \mbox{ for } \, 0 \leq i,j \leq d-1, \\
c_i x_i y_1 \equiv x_{i+1} y_0 \ \ (\text{mod } \mathcal{I}) \quad & \mbox{ for } \, -1 \leq i \leq d-1,
\end{align*}
where $c_{i}$ is a non-zero constant for every $i$ and $x_{-1} = x_d := 0$. Then it is a direct computation to show that $\mathrm{HF}(T/\mathcal I, d) = \mathrm{HF}(T/\mathcal I, d+1) = d+2$ and $\mathrm{HF}(T/\mathcal I, d+2) = d+3$. Moreover, one has
\[
\left(T/\mathcal I\right)_{d+2} = \langle y_0^{d+2}, y_0^{d+1}y_1,\ldots, y_1^{d+2}\rangle = S^{d+2}\langle y_0,y_1\rangle.
\] 
Thus if $\mathcal{J} \in \underline{\mathrm{VSP}}(G_d,d+2)$ then $\mathcal J$ contains $\mathcal I$ and an ideal $\mathcal Q$ generated by some element $q\in \left(T/\mathcal I\right)_{d+2}$.  

For the converse, let $\mathcal J_{\mathcal{Q}} = \mathrm{Ann}(G_d)_{\leq d-1}+ \mathcal Q$, for some $\mathcal Q=\langle q\rangle$ such that $q \in S^{d+2} \langle y_0,y_1 \rangle$. Note that $\mathrm{HF}(T / \mathcal{J}_{\mathcal{Q}},d+2) = d+2$. We can apply Gotzmann's Persistence Theorem \cite[Theorem 3.8]{Green} to conclude that $\mathrm{HF}(T / \mathcal{J}_{\mathcal{Q}},j) = d+2$ for all $j \geq d+2$. 

To show that all such $\mathcal{J}_{\mathcal{Q}}$ are in fact in $\underline{\mathrm{VSP}}(G_d,d+2)$, consider $\mathcal{Q} = \langle q \rangle$ where $q$ is a form with $d+2$ distinct roots. This gives us $d+2$ distinct points $z_1,\ldots, z_{d+2}\in \PP(\langle y_0,y_1 \rangle^{*})$. By Lemma~\ref{borderankGd} and Corollary~\ref{tangentcone1}, we may find a border rank decomposition defining a family of zero-dimensional schemes $R(t)$, whose limiting scheme is the smooth scheme $Z$ supported at $z_1,\ldots, z_{d+2}$. 

 Let $\mathcal I_{Z}$ be the radical ideal defining the limiting scheme $Z$. By Theorem \ref{correspondence}, there is a corresponding ideal $\mathcal J'\in  \underline{\mathrm{VSP}}(G_d,d+2)$ such that $\mathcal J'^{sat} = \mathcal I_{Z}$. Note that $\mathcal I_{Z} = \mathcal J_{\mathcal Q}^{sat}$. 

Since $\mathcal{J}'\in \underline{\mathrm{VSP}}(G_d,d+2)$, by the first part of this proof, $\mathcal{J}' = \mathcal{J}_{\mathcal{Q'}}$, for some $\mathcal Q' = \langle q'\rangle$. Therefore one has $\mathcal J_{\mathcal Q'}^{sat} = \mathcal I_{Z}$.  Since $\mathcal J_{\mathcal Q'}^{sat} = \mathcal J_{\mathcal Q}^{sat}$, we must have $\mathcal Q = \mathcal Q'$. In conclusion, we derive $\mathcal J_{\mathcal Q} = \mathcal J' \in \underline{\mathrm{VSP}}(G_d,d+2)$. 

Now, consider the morphism

\begin{align*}
\psi_d: \underline{\mathrm{VSP}}(G_d,d+2)&\longrightarrow   \PP((T / \mathcal{I})_{d+2}), \\
  \mathcal{J}_{\mathcal{Q}} & \longmapsto  [q].
\end{align*}
This morphism is proper and hence closed. We have shown that the generic point of $\PP((T / \mathcal{I})_{d+2})$ lies in the image, therefore $\psi_d$ is surjective. 
\end{proof}
\end{proposition}

Let $S^{d+2} \PP^1$ denote the $(d+2)$-fold symmetric product of the projective line. Proposition \ref{structureidealVSPGd} yields the following:

\begin{theorem}\label{vspprojectivespaces}
The projective variety $\underline{\mathrm{VSP}}(G_d, d+2)$ is isomorphic to the projective space $\PP^{d+2}\cong \PP(S^{d+2}\CC^2)\cong S^{d+2} \PP^1$.
\begin{proof}
The morphism $\psi_d$ in the proof of Proposition \ref{structureidealVSPGd} is surjective. It is also injective because the point $[q]$ uniquely determines the ideal $\mathcal{I}_{\mathcal{Q}}$. Since $\PP^{d+2}$ is smooth and so normal, the map $\psi_d$ is an isomorphism by a variant of Zariski's Main Theorem \cite[Corollary 4.6]{Liu}. The isomorphism follows from the description of the vector space $\left(T/\mathcal I\right)_{d+2}$ given in the proof of Proposition \ref{structureidealVSPGd}. 
\end{proof}
\end{theorem}

\section{An infinite series of wild cubics and their $\underline{\mathrm{VSP}}$}\label{wildcubics}

For every $k\geq 1$ and $n=3k+1$, we introduce the following infinite series of concise cubic forms: 

\[
F_n = x_0x_1^2 + x_1x_2x_4 + x_3x_4^2+ x_4x_5x_7 + x_6x_7^2 + 
\]
\[
+ x_8x_7x_{10} + x_9x_{10}^2 +\cdots + x_{n-4}x_{n-3}^2 + 
\]
\[
+ x_{n-3}x_{n-2}x_n + x_{n-1}x_n^2.
\]
\vspace{2mm}
\begin{remark}
This infinite series is inspired by the examples appeared in \cite{gr15}, which in turn 
are a generalization of the {\it Perazzo cubic hypersurface} $\lbrace F_4 = 0\rbrace\subset \PP^4$. The latter is exactly
the wild cubic found in \cite[\S 4]{bb15}. 
\end{remark}

Let $\mathrm{Ann}(F_n)\subset T = \CC[y_0,\ldots, y_n]$ be the annihilator of $F_n$ with $y_i$ being dual to $x_i$.

\begin{remark}
Note that $F_4$ coincides with $G_3$ from \S \ref{higherdeg} up to change of basis. So $\underline{\mathrm{VSP}}(F_4,5)\cong \PP^5$. 
\end{remark}

The following combinatorial arrangement of lines is important for us to study the infinite series $F_n$; Proposition \ref{chainlines}
provides the motivation for looking at it. 

\begin{definition}[{\bf Chains of lines}]\label{chain}
A {\it chain of lines} is a collection of distinct lines $C^{1},\ldots, C^{m} = \PP^1$ such that (up to reindexing) for $1\leq i,j\leq m$: 
\[
C^{i} \cap C^{i+1} \neq \emptyset, \mbox{ and }
\]
\[
C^i\cap C^j = \emptyset \mbox{ otherwise }. 
\] 
\end{definition}

\begin{proposition}\label{chainlines}
The projective scheme whose ideal is $\left(\mathrm{Ann}(F_n)_2\right)^{sat}$ is a chain of lines.
\begin{proof}
Let $\mathcal I = \mathrm{Ann}(F_n)_2$ and let $\mathcal I^{sat}$ denote its saturation. Let $n=3k+1$ for $k\geq 1$. 
We divide the proof according to the residue $(\mathrm{mod}\ 3)$ of each index $0\leq j\leq n$ in $y_j$. \\

\begin{itemize}
\item $j\equiv 0 \ (\mathrm{mod} \ 3)$. The monomial $y_jy_i\in \mathcal I$ for all $i\neq j+1$. Moreover,
$y_jy_{j+1}^3\in \mathcal I$, therefore $y_j\in \mathcal I^{sat}$. \\

\item $j\equiv 2 \ (\mathrm{mod} \ 3)$. In this case, one has $y_jy_i\in \mathcal I$ for all $i\neq j-1,j+2$. 
Note that $y_jy_{j-1}^2, y_jy_{j+2}^2\in \mathcal I$.  Therefore $y_j\in \mathcal I^{sat}$. \\

\item $j\equiv 1 \ (\mathrm{mod} \ 3)$. In this case, $y_jy_{j+3h}\in \mathcal I\subset \mathcal I^{sat}$ for $h\geq 2$. 

For $i\equiv 1 \ (\mathrm{mod} \ 3)$, $y_i\notin \mathcal I^{sat}$. We show that $y_i^k\notin \mathcal I$ for any $k\geq 2$. 
Assuming on the contrary that $y_i^k\in \mathcal I$, one has $y_i^k = \sum_{j=0}^s m_j h_j$, where the $h_j$ are the generators of $\mathcal I$ and $m_j\in T$. 
However, on the right-hand side, every monomial that is divisible by $y_i$ is divisible by some other distinct $y_j$ as well. 

We show that $y_iy_{i+3}\notin \mathcal I^{sat}$. On the contrary, suppose $y_iy_{i+3}\in \mathcal I^{sat}$. Thus, there exists $k>1$ such that $\left(y_iy_{i+3}\right)^k\in \mathcal I$. Using the same argument as above, we see that every monomial in $\mathcal I$, that is divisible by $y_iy_{i+3}$, must be divisible by some other distinct $y_j$ as well. 
\end{itemize}

In conclusion, the saturated ideal $\mathcal I^{sat}$ is generated by: 
\[
\mathcal I^{sat} = \left\langle y_{3h}, y_{3h+2}, y_{3h+1}y_{3(h+s)+1}, h\geq 0, s\geq 2\right\rangle
\]

The projective scheme defined by $\mathcal I^{sat}$ is a chain of $k$ lines $C_k$, as in Definition \ref{chain}. Its irreducible components are lines $L_{i,j}$, such that $|i-j|=3$ and  $i\equiv 1 \ (\mathrm{mod} \ 3)$, whose ideal is defined by 
\[
\mathcal J_{i,j} = \left\langle y_k \ | \ k\neq i, j\right\rangle. 
\]
This concludes the proof. 
\end{proof}
\end{proposition}

\noindent {\it Notation.} In the following, let $C_k$ denote the chain of lines appearing in the proof of Proposition \ref{chainlines}, corresponding to the cubic $F_n$ when $n=3k+1$. Moreover, $C_k$ comes equipped with an ordering on its components $C_k^{h}$ induced by the lexicographic order on the variables $y_1> y_4 > \cdots > y_{3k+1}$. Therefore, the $h$-th component of $C_k$ refers to the line where the homogeneous coordinates are $y_{3h-2}$ and $y_{3h+1}$.

\begin{proposition}\label{borderankFn}
Let $n\geq 4$. The form $F_n$ has minimal border rank $n+1$. 
\begin{proof}
By conciseness, $\underline{{\bf r}}(F_n)\geq n+1$. We show the opposite inequality as follows. 
Let $n=3k+1$ and consider the powers of linear forms
\[
\ell_{1,1}^3, \ \ \ell_{1,2}^3, \ \  \ell_{1,3}^3, \ \ \ell_{1,4}^3,
\]
\[
 \ \ \ \ \ \ell_{2,2}^3, \ \ \ell_{2,3}^3, \ \ \ell_{2,4}^3, 
\]
\[
\cdots \ \ \cdots \ \ \cdots
\]
\[
 \ \ \ \ \ \ell_{k-1,2}^3, \ \ \ell_{k-1,3}^3, \ \ \ell_{k-1,4}^3, 
\]
\[
\ell_{k,1}^3, \ \ \ell_{k,2}^3, \ \  \ell_{k,3}^3, \ \ \ell_{k,4}^3,
\]
where $\ell_{i,j}$ is a linear form defined on the $i$-th component of the chain of lines $C_k$, i.e. $\ell_{i,j}$ depends only on
corresponding two variables. Note that there are $n+1$ of such forms. Moreover, we require $\ell^3_{i,j}\neq x_{3i-2}^3, x_{3i+1}^3$ for $1\leq i\leq k-1$, where the latter cubes correspond to the intersections between the $i$-th component of $C_k$ and the other lines in $C_k$. 
This may be regarded as a configuration of points on the lines $C_k^{i}$ of the chain $C_k$; an instance of this is depicted in Figure \ref{fig:config}. 

Write $\PP\widehat{T}_{\ell^3}$ for the affine cone of the Zariski tangent space to $\nu_3(\PP^n)$ at the point $\ell^3$. As in the proof of Proposition \ref{borderankGd},
for $d=3$, we have

\[
F_4 = x_0x_1^2 + x_1x_2x_4 + x_3x_4^2 \in \left\langle \PP\widehat{T}_{\ell_{1,j}^{3}} \mbox{ for } 2\leq j\leq 4 \right\rangle. 
\]

More generally, for each $1\leq i\leq k$, one has: 
\[
H_i = x_{3i-3}x_{3i-2}^2 + x_{3i-2}x_{3i-1}x_{3i+1} + x_{3i}x_{3i+1}^2\in  \left\langle \PP\widehat{T}_{\ell_{i,j}^{3}} \mbox{ for } 2\leq j\leq 4 \right\rangle.
\]
Thus $F_n  \in \left\langle \PP\widehat{T}_{\ell_{i,j}^{3}} \mbox{ for } 1\leq i\leq k, 2\leq j\leq 4 \right\rangle$.

The forms $\ell_{i,j}^3$ are linearly dependent. Indeed, consider the forms on the first component $C_k^{1}$ of the chain:  
$\ell_{1,1}^3, \ell_{1,2}^3, \ell_{1,3}^3$ and $\ell_{1,4}^3$ are linearly independent and span every cubic form defined on $C_k^{1}$. 
Thus the cube $x_4^3$ may be written as 
\[
x_4^3 = \sum_{j=1}^4 \lambda_j \ell_{1,j}^3.  
\]
On the line $C_k^{2}$, we have the linear relation: 
\[
\mu_1 x_4^3 + \left(\sum_{j=2}^4 \mu_j \ell_{2,j}^3\right) + \mu_5 x_7^3 = 0. 
\]
So we may rewrite the cube $x_7^3$ as a linear combination of the $\ell_{i,j}^3$ with $1\leq i\leq 2$. Proceed similarly up to $C_k^{k-1}$, where
we express $x_{3k-2}^3$ as a linear combination of $\ell_{i,j}^3$ with $1\leq i\leq k-1$. Finally, on the line $C_k^k$ use the linear relation among
the five cubic forms $x_{3k-2}^3, \ell_{k,1}^3, \ell_{k,2}^3, \ell_{k,3}^3$, and $\ell_{k,4}^3$. This process gives a linear relation among the $\ell_{i,j}^3$.

Now, to conclude, use Proposition \ref{tangentcones} which yields $\underline{{\bf r}}(F_n)\leq n+1$. 
\end{proof}
\end{proposition}

\begin{center}
\begin{figure}
\caption{The chain of lines $C_k$ and a configuration of points featured in the proof of Proposition \ref{borderankFn}, for $k=4$ and $n=13$.}
\label{fig:config}
\includegraphics[scale=0.5]{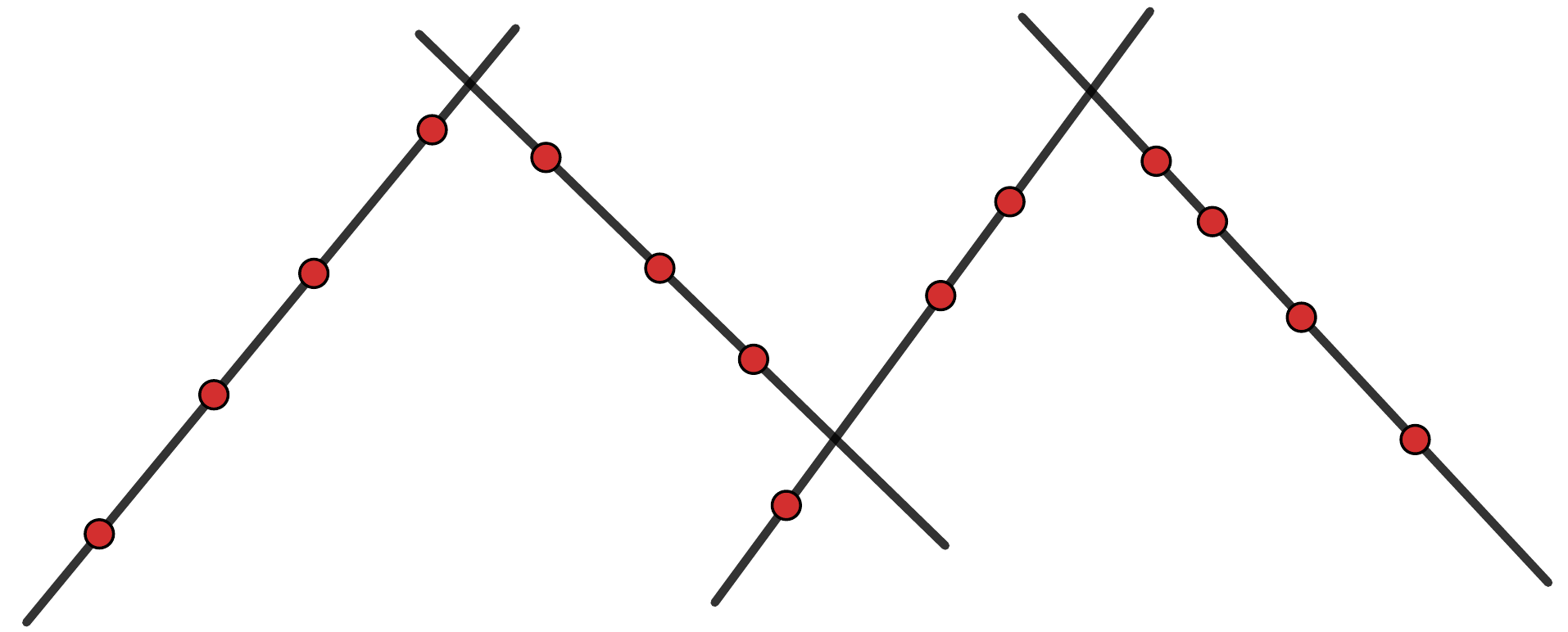}
\end{figure}
\end{center}

\begin{remark}
For every $k\geq 1$ and $n=3k+1$, one has 
\[
n+1 = \underline{{\bf r}}(F_n) < \mathrm{cr}(F_n)\leq \mathrm{sr}(F_n)\leq 6k = 2(n-1),
\] 
as every $2$-jet is smoothable and so is their union. 
\end{remark}

\begin{corollary}\label{infiniteseries}
The cubics $F_n$ are wild. 
\begin{proof}
Their Hessian is vanishing, as the partial derivatives are algebraically dependent. By Proposition \ref{borderankFn}, $F_n$ has minimal 
border rank. Theorem \ref{hessvswild} shows that the cubics $F_n$ are wild. 
\end{proof}
\end{corollary}

\begin{lemma}\label{psiinj}

Let $n=3k+1$ and let $C_k$ be the chain of lines from Proposition \ref{chainlines}. Let $\mathcal J\in  \underline{\mathrm{VSP}}(F_n, n+1)$. 
Then: 
\begin{enumerate}
\item[(i)] there exist forms $q_1,\ldots, q_s$ in the variables $y_{3h+1}$'s such that $\mathcal J = \mathrm{Ann}(F)_2 + \langle q_1,\ldots, q_s\rangle$; 

\item[(ii)] the ideal $\mathcal J^{sat}$ defines a projective scheme that is a zero-dimensional scheme of length $n+1$ supported on $C_k$;

\item[(iii)] each form $q_j$ has degree at most $n+1$; 

\item[(iv)] there exists a proper map 
\[
\psi_k: \underline{\mathrm{VSP}}(F_n, n+1)\rightarrow \mathrm{Hilb}_{n+1}(C_k),
\]
where the latter is the Hilbert scheme of zero-dimensional schemes of length $n+1$ supported on the reducible curve $C_k\subset \PP^{n}$. 
\end{enumerate}
\begin{proof}
(i). From the description of $\mathcal I = \mathrm{Ann}(F)_2$ in the proof of Proposition \ref{chainlines}, it is straightforward to see that each degree $d\geq 4$ graded piece of the quotient ring $T/\mathcal I$ has a basis of monomials $\lbrace m_1,\ldots, m_h\rbrace$, where each $m_j$ is a monomial in two variables $y_{3h-2}$ and $y_{3h+1}$; each such a pair of variables corresponds to a unique line in the chain $C_k$. Since a necessary condition for membership of an ideal $\mathcal J$ in $\underline{\mathrm{VSP}}(F_n, n+1)$ is possessing a generic Hilbert function, we add forms $q_j$ (in the variables $y_{3h+1}$'s) to $\mathcal J$ until the ideal reaches a generic Hilbert function. \\
(ii). Let $\mathcal I = \mathrm{Ann}(F)_2$. Note that the Hilbert polynomial of $\mathcal J^{sat}$ is $n+1$. Moreover, $\mathcal J^{sat}\supset \mathcal{I}^{sat}$. By Proposition \ref{chainlines}, $\mathcal{I}^{sat}$ is the ideal defining the projective scheme $C_k$. Therefore $\mathcal J^{sat}$ defines a
zero-dimensional scheme of length $n+1$ supported on $C_k$. \\
(iii). The Castelnuovo-Mumford regularity of a zero-dimensional scheme of length $n+1$ is at most $n+1$ \cite[Theorem 1.69]{ik}, 
and the degrees of the generators $q_j$ are bounded above by the regularity. \\
(iv). By (ii), 
the projective scheme defined by $\mathcal J^{sat}$ is a zero-dimensional scheme of length $n+1$ supported on $C_k$. Define the map:
\[
\psi_k: \underline{\mathrm{VSP}}(F_n, n+1)\rightarrow \mathrm{Hilb}_{n+1}(C_k),
\]
\[
\mathcal J \mapsto \mathcal J^{sat}. 
\]
Then this is a well-defined morphism. It is proper because it is projective.
\end{proof}
\end{lemma}

\begin{theorem}\label{vspreducible}
Let $n=3k+1\geq 10$. The variety $\underline{\mathrm{VSP}}(F_n, n+1)$ is reducible. 
\begin{proof}
Since $\psi_k$ is proper by Lemma \ref{psiinj}(iv), if $\underline{\mathrm{VSP}}(F_n, n+1)$ were irreducible, then the image $\psi_k\left(\underline{\mathrm{VSP}}(F_n, n+1)\right)$ would be closed and irreducible. We show next that $\psi_k\left(\underline{\mathrm{VSP}}(F_n, n+1)\right)$ has at least two irreducible components.

As in proof of Proposition~\ref{borderankFn}, $F_n$ is in the span of the affine cones of Zariski tangent spaces of $\nu_3(\PP^n)$ at the following points:
\[
 \ \ \ \ \ \ell_{1,2}^3, \ \  \ell_{1,3}^3, \ \ \ell_{1,4}^3,
\]
\[
 \ \ \ \ \ \ell_{2,2}^3, \ \ \ell_{2,3}^3, \ \ \ell_{2,4}^3, 
\]
\[
\cdots \ \ \cdots \ \ \cdots
\]
\[
\ell_{a,1}^3, \ \ \ell_{a,2}^3, \ \  \ell_{a,3}^3, \ \ \ell_{a,4}^3,
\]
\[
\cdots \ \ \cdots \ \ \cdots
\]
\[
\ell_{b,1}^3, \ \ \ell_{b,2}^3, \ \ \ell_{b,3}^3, \ \ \ell_{b,4}^3,
\]
\[
\cdots \ \ \cdots \ \ \cdots
\]
\[
\ \ \ \ \ \ \ell_{k,2}^3, \ \  \ell_{k,3}^3, \ \ \ell_{k,4}^3.
\]
Here we generalize the original configuration described in Proposition \ref{borderankFn}, where the forms $\ell_{1,1}^3$ and $\ell_{k,1}^3$ are replaced by $\ell_{a,1}^3$ and $\ell_{b,1}^3$, with $1 \leq a < b \leq n-3$ and $a \equiv b \equiv 1 \ (\text{mod } 3)$. Note that, if $n \geq 10$, there are at least two such pairs $(a,b)$. 
To see that the forms $\ell_{i,j}^3$ above are linearly dependent, perform the same procedure presented in the proof of Proposition \ref{borderankFn}, 
starting from the $a$-th component and ending at the $b$-th component of $C_k$. This produces a linear relation among the $\ell_{i,j}^3$ with $a\leq i\leq b$. 

Thus, by Corollary \ref{tangentcone1} and Theorem \ref{correspondence}, we may find a border rank decomposition given by a family of schemes $R(t)$ over the base $\mathrm{Spec}(\CC[t^{\pm}])$, whose limiting scheme is supported on the points $\ell_{i,j}^3$ above, with a corresponding ideal $\mathcal J_{(a,b)}\in \underline{\mathrm{VSP}}(F_n, n+1)$. Thus $\psi_k(\mathcal J_{(a,b)})\in \mathrm{Hilb}_{n+1}(C_k)$. 

As before, let $C_k^{h}$ denote the $h$-th component of $C_k$. Let $\mathrm{Hilb}_{n+1}(C_k)^{(a,b)}$ be the irreducible component of $\mathrm{Hilb}_{n+1}(C_k)$
defined by: 

\[
\mathrm{Hilb}_{n+1}(C_k)^{(a,b)} = \overline{\lbrace Z\subset C_k \ | \ Z \mbox{ smooth, } Z\cap C_k^{a} = Z\cap C_k^{b} = 4, Z\cap C_k^{j}=3, j\neq a,b\rbrace}. 
\]
\vspace{1mm}

Consider the components $\mathrm{Hilb}_{n+1}(C_k)^{(a,b)}$. Notice that $\dim \mathrm{Hilb}_{n+1}(C_k)^{(a,b)} = 3k+2 = n+1$. To see this, let 
$Z\in \mathrm{Hilb}_{n+1}(C_k)^{(a,b)}$ be a general point and so $Z$ is a smooth zero-dimensional scheme. The normal bundle $N_{Z/C_k}$ has $n+1$ global sections (an affine coordinate at each smooth point of $Z$). Therefore 
\[
n+1 = h^0(N_{Z/C_k}) = \dim T_{Z} \ \mathrm{Hilb}_{n+1}(C_k) =   
\]
\[
 = \dim T_{Z} \ \mathrm{Hilb}_{n+1}(C_k)^{(a,b)}
= \dim \mathrm{Hilb}_{n+1}(C_k)^{(a,b)}.
\]

Next, we show that the components $\mathrm{Hilb}_{n+1}(C_k)^{(1,2)}$ and $\mathrm{Hilb}_{n+1}(C_k)^{(1,3)}$ are distinct. 
The general element $Z\in \mathrm{Hilb}_{n+1}(C_k)^{(1,2)}$ is such that $Z\cap C_k^{2}$ is a smooth scheme of length four supported
outside $C_k^{1}\cap C_k^{2}$ and $C_k^{2}\cap C_k^{3}$. 

However, any flat limit $R$ of zero-dimensional schemes $R(t) \in \mathrm{Hilb}_{n+1}(C_k)^{(1,3)}$,
with $\mathrm{length}(R\cap C_k^{2})=4$, is such that 
\[
\mbox{ \emph{either} } \ \mathrm{Supp}(R)\cap C_k^{1}\cap C_k^{2} \neq \emptyset \ \mbox{ {\it or} } \ \mathrm{Supp}(R)\cap C_k^{2}\cap C_k^{3} \neq \emptyset. 
\]
Thus $Z$ cannot be in the component $\mathrm{Hilb}_{n+1}(C_k)^{(1,3)}$. 

Since $\mathrm{Hilb}_{n+1}(C_k)^{(1,2)}$ and $\mathrm{Hilb}_{n+1}(C_k)^{(1,3)}$ have the same maximal dimension, they are two distinct irreducible components of the Hilbert scheme $\mathrm{Hilb}_{n+1}(C_k)$. As the image $\psi_k(\underline{\mathrm{VSP}}(F_n, n+1))$ is closed, it contains both of them; therefore the image is reducible. In conclusion, $\underline{\mathrm{VSP}}(F_n, n+1)$ must be reducible. 
\end{proof}
\end{theorem}

\begin{remark}
The variety $\underline{\mathrm{VSP}}(F_7, 8)$ is  at least $8$-dimensional. Let $\mathcal I = \mathrm{Ann}(F)_2$. One can check there is a generically injective map 
\[
\rho: \underline{\mathrm{VSP}}(F_7, 8) \longrightarrow \PP\left((T/\mathcal I)_4\right)= \PP^8, 
\]
\[
\mathcal J\longmapsto \mathcal J_4,
\]
i.e. the generic fiber of $\rho$ is a single point. We do not know whether $\underline{\mathrm{VSP}}(F_7, 8)$ is irreducible or not. The locus where the morphism is injective is then birational to $\PP^8$. If it is irreducible and $8$-dimensional, it cannot be isomorphic to $\PP^8$: the isomorphic fibers $\rho^{-1}(x_1^4)$ and $\rho^{-1}(x_7^4)$ both contain a linear space $\PP^4$. Thus these two linear spaces do not intersect. 
\end{remark}

\begin{question}
Is $\underline{\mathrm{VSP}}(F_7, 8)$ irreducible? Are the irreducible components of the border varieties $\underline{\mathrm{VSP}}(F_n, n+1)$ rational? More generally, it would be interesting to analyze rationality and unirationality of (the irreducible components of) border varieties of sums of powers $\underline{\mathrm{VSP}}$'s alike in the context of $\mathrm{VSP}$'s; see for instance \cite{mm} for several results in this direction. 
\end{question}

\end{document}